\documentclass[a4paper,12pt]{amsart}
\usepackage{amsaddr}
\usepackage[a4paper,left=24mm,right=24mm,top=34mm,bottom=34mm]{geometry}
\usepackage[utf8]{inputenc}
\usepackage{color}
\usepackage{xr-hyper} 
\usepackage[backref]{hyperref}
\usepackage{cleveref}
\usepackage{comment}
\usepackage[shortlabels]{enumitem}
\usepackage{todonotes}
\usepackage{multirow}
\hypersetup{
	colorlinks,
	citecolor=red,
	filecolor=blue,
	linkcolor=blue,
	urlcolor=purple
}


\usepackage[all]{xy}
\usepackage{amsmath,amssymb,amsfonts,amsthm,graphicx}

\usepackage{multirow}
\usepackage{comment}
\theoremstyle{plain}

\tikzset{node distance=2cm, auto}

\newtheorem{thm}{Theorem}[section]
\newtheorem{conj}[thm]{Conjecture}
\newtheorem{cor}[thm]{Corollary}
\newtheorem{prop}[thm]{Proposition}

\newtheorem{lm}[thm]{Lemma}
\newtheorem{example}[thm]{Example}

\theoremstyle{definition}
\newtheorem{df}[thm]{Definition}
\newtheorem{notation}[thm]{Notation}

\theoremstyle{remark}
\newtheorem{rem}[thm]{Remark}

\newcommand{\N}{\operatorname{N}}
\newcommand{\Z}{\operatorname{Z}}

\newcommand{\C}{\operatorname{C}}

\newcommand{\SL}{\operatorname{SL}}
\newcommand{\SU}{\operatorname{SU}}

\newcommand{\bS}{{\mathbf S}}
\newcommand{\bT}{{\mathbf T}}
\newcommand{\NNN}{\operatorname{N}}

\newcommand{\wt}{\widetilde}

\usepackage{cleveref,enumitem}
\newlist{asslist}{enumerate}{1} % also creates a counter called 'propenumi'
\setlist[asslist]{label=(\roman*), ref=\thethm(\roman*)}

\crefalias{asslistenumi}{Assumption} 

\newlist{thmlist}{enumerate}{1} % also creates a counter called 'propenumi'
\setlist[thmlist]{label=(\alph*), ref=\thethm(\alph*)}

\newcommand{\Sp}{\operatorname{Sp}}

\newcommand{\bl}{\operatorname{bl}}

\newcommand{\Irr}{\operatorname{Irr}}

\newcommand{\Aut}{\operatorname{Aut}}
\newcommand{\Ind}{\operatorname{Ind}}
\newcommand{\Res}{\operatorname{Res}}

\newcommand{\ol}{\overline}

\newcommand{\wh}{\widehat}

\newcommand{\bC}{{{\mathbf C}}}
\newcommand{\bG}{{{\mathbf G}}}
\newcommand{\bH}{{{\mathbf H}}}
\newcommand{\bW}{{{\mathbf W}}}
\newcommand{\bV}{{{\mathbf V}}}
\numberwithin{equation}{section}

\excludecomment{commentMe}

\title{On the Alperin--McKay conjecture for $2$-blocks of maximal defect}

\begin{document}
\begin{abstract}
In this paper, we show that the Alperin--McKay conjecture holds for $2$-blocks of maximal defect. A major step in the proof is the verification of the inductive Alperin--McKay condition for the principal $2$-block of groups of Lie type in odd characteristic.
\end{abstract}

\date{\today}
\author{Julian Brough and Lucas Ruhstorfer}
%\thanks{The first author thanks the Humbold Foundation for its support.}
\thanks{}
\thanks{}
\address{School of Mathematics and Natural Sciences University of Wuppertal, Gau\ss str. 20, 42119 Wuppertal, Germany \newline
Fachbereich Mathematik, TU Kaiserslautern, 67653 Kaiserslautern, Germany}
\email{brough@uni-wuppertal.de,\, ruhstorfer@mathematik.uni-kl.de}

\maketitle

\section{Introduction}

In the representation theory of finite groups some of the most important conjectures predict a very strong relationship between the representations of a finite group $G$ and certain representations of its $\ell$-local subgroups, where $\ell$ is a prime dividing the order of $G$. One of these conjectures is the Alperin--McKay conjecture.
For an $\ell$-block $b$ of $G$ we denote by $\Irr_0(G,b)$ its set of height zero characters.

\begin{conj}[Alperin--McKay]
	Let $b$ be an $\ell$-block of $G$ with defect group $Q$ and $B$ its Brauer correspondent in $\mathrm{N}_G(Q)$. Then $$|\Irr_0(G,b)|= |\Irr_0(\mathrm{N}_G(Q),B)|.$$
\end{conj}

In this article we show that the Alperin--McKay conjecture holds for $2$-blocks of maximal defect.

\begin{thm}\label{main}
Let $b$ be a $2$-block of a finite group $G$ whose defect group is a Sylow $2$-subgroup $Q$ and $B$ its Brauer correspondent in $\mathrm{N}_G(Q)$. Then $$|\Irr_0(G,b)|= |\Irr_0(\mathrm{N}_G(Q),B)|.$$
\end{thm}

Späth \cite[Theorem C]{IAM} showed that the Alperin--McKay conjecture holds for the prime $\ell$ if the so-called inductive Alperin--McKay condition holds for all blocks of all finite quasi-simple groups with respect to the prime $\ell$. It is therefore possible to approach the Alperin--McKay conjecture through the classification of finite simple groups. Thanks to the work of several authors the
inductive condition has been shown for all finite simple groups except in the case where $G$ is a group of Lie type defined over a field of characteristic $p \neq \ell$. Hence, we will focus on the case where $G$ is a group of Lie type defined over a field of odd characteristic. In their seminal paper, Malle--Späth \cite{MS} showed that in this case $G$ is McKay-good for the prime $2$. For this they constructed a bijection $\Irr_{2'}(G) \to \Irr_{2'}(M)$ between the set of irreducible odd degree characters of $G$ and the corresponding set of characters of a well chosen subgroup $M$ of $G$ containing $\N_G(Q)$, for $Q$ a Sylow $2$-subgroup of $G$. Based on their bijection we are able to construct an explicit bijection between the height zero characters in the  principal blocks of $G$ and $\N_G(Q)$ and show the following.

\begin{thm}\label{principal}
Let $G$ be a quasi-simple group of Lie type defined over a field of odd characteristic. Then the principal $2$-block of $G$ satisfies the AM-condition.
\end{thm}

In a previous article, the second author has reduced the verification of the inductive  Alperin--McKay condition to quasi-isolated blocks of $G$ \cite{Jordan2} and then subsequently for groups of type $A$ to unipotent blocks \cite{Jordan3}. 
One major hurdle that arises when making use of this reduction, in its current form, is that the possibility to choose a suitable subgroup $M$ as done in Malle-Späth no longer holds.
As a consequence of the bijection explicitly constructed to prove Theorem \ref{principal} and the classification of quasi-isolated elements we obtain the following:

\begin{cor}\label{class}
	Let $G$ be a quasi-simple group of classical Lie type. Then every $2$-block of $G$ satisfies the AM-condition.
\end{cor}

Unfortunately, if $G$ is a group of Lie type with exceptional root system there are many quasi-isolated $2$-blocks. However, one can show that the principal $2$-block is the unique quasi-isolated $2$-block of maximal defect.

\begin{cor}\label{exc}
	Let $G$ be a quasi-simple group of exceptional Lie type. Then every $2$-block of maximal defect of $G$ satisfies the AM-condition.
\end{cor}

Using this we are able to settle the  AM-condition for blocks of maximal defect of finite quasi-simple groups, which is then enough to establish Theorem \ref{main}.

\subsection*{Structure of the paper}
In Section 2 we derive some fundamental results on the structure of normalisers of Sylow $2$-subgroups of groups of Lie type. This will be used in Section 3 to provide a description of the height zero characters in the principal block of this normaliser. In the same section we moreover give a parametrisation of the height zero characters of the principal block of $G$ in terms of the $1$-Harish-Chandra series. In Section 4 and Section 5 we study the action of group automorphisms of $G$ on our parametrisation of characters. This will be used in Section 6 to prove the AM-condition for the principal block. In Section 7 we deal with the remaining finite simple groups and in Section 8 we prove our main results. 

\subsection*{Acknowledgments} The authors would like to thank both Britta Späth and Gunter Malle for discussions relating to previous drafts of this article. This paper is a contribution via the second author to the SFB TRR 195, and the first author is supported by the DFG (Project: BR 6142/1-1).

\section{Sylow $2$-subgroups}

\subsection{Weyl groups}

It is well known that the Sylow $2$-subgroups of the symmetric group are self-normalising.
That is for $P\in {\rm Syl}_2(\mathfrak{S}_n)$, we have that $\NNN_{\mathfrak{S}_n}(P)=P$.
It turns out for all Weyl groups of irreducible type that the Sylow $2$-subgroups will be self-normalising.
In the following we denote by $\C_n$ the cyclic group of order $n$, while $C_n$ denotes the root system of type $C$ with $n$ nodes.

\begin{lm}\label{self normalising}
Let $W$ be a Weyl group. Then every Sylow $2$-subgroup of $W$ is self-normalising.
\begin{proof}
Since any Weyl group is a direct product of irreducible Weyl groups we can assume that $W$ is irreducible.
The case $W(A_n)\cong \mathfrak{S}_{n+1}$ is well-known, which moreover implies the case $W(C_n)\cong \C_2 \wr \mathfrak{S}_n$.
The group $W(D_n)$ also follows from the symmetric group as it arises as a normal subgroup of index 2 in $W(C_n)$ isomorphic to $\C_2^{n-1}\rtimes \mathfrak{S}_n$ (which can be constructed as the quotient of $W(C_n)$ by the kernel of the homomorphism $\C_2^n\rightarrow \C_2$ which maps $(g_1,\dots,g_n)$ to the product $g_1\dots g_n$).
This only leaves the exceptional cases. The result is immediate for $W(G_2)\cong {\rm Dih}_{12}$.
For the remaining cases the description of these groups provided in \cite[Section 2.12]{HumphReflec} will be taken.

Observe that $W(F_4)$ arises as the semidirect product of $W(D_4)$ with the automorphism group of the Dynkin diagram of type $D_4$.
The group $W(D_4) \cong (\C_2)^3 \rtimes \mathfrak{S}_4$ is generated by signed permutations $g_1=(1,2)(-1,-2)$ $g_2=(2,3)(-2,-3)$, $g_3=(3,4)(-3,-4)$ and $g_4=(3,-4)(-3,4)$.
Set $\gamma_1$ to be the automorphism of order $2$ fixing both $g_1,g_2$ and interchanging $g_3$ and $g_4$, while $\gamma_2$ denotes the automorphism of order $3$ which fixes $g_2$ and permutes $g_1,g_3$ and $g_4$ cyclically.
Then $W(F_4)\cong W(D_4)\rtimes \langle \gamma_1,\gamma_2\rangle$.
The group $W(D_4)$ has three Sylow $2$-subgroups one of which must be fixed by $\gamma_1$.
Moreover only one Sylow $2$-subgroup of $W(D_4)$ contains $g_2$ and thus all three subgroups are fixed by $\gamma_2$.
In particular, $W(D_4)$ has a Sylow $2$-subgroup $Q$ which is fixed by both automorphisms $\gamma_1$ and $\gamma_2$.
Set $P:=\langle Q,\gamma_1\rangle$ which is a Sylow $2$-subgroup of $W(F_4)$.
As $\N_{W(F_4)}(Q)=\langle Q,\gamma_1,\gamma_2\rangle$ and $P^{\gamma_2}=\langle Q, \gamma_1^{\gamma_2}\rangle\ne P$, it follows that $\N_{W(F_4)}(P)=P$.

The group $W(E_6)$ contains a subgroup $W^+(E_6)\cong \SU_4(2)$ of index two.
In $\SU_4(2)$ the normaliser of a Sylow $2$-subgroup $Q$ is a Borel subgroup $B$, but $B=Q$ as $q=2$.
Hence $W^+(E_6)$ and thus $W(E_6)$ has self-normalising Sylow $2$-subgroups.
The same argument proves the case of $W(E_7)\cong \C_2\times  \Sp_6(2)$.
While for $W(E_8)$ the index two subgroup $W^+(E_8)$ surjects onto $\Omega_8^+(2)$ with kernel $Z(W(E_8))$ of order $2$.
Thus for $G$ the universal cover of $\Omega_8^+(2)$ with $Z(G)\cong \C_2\times \C_2$, the same argument as used in $E_6$ shows that $G$ and consequently also the groups $\Omega_8^+(2)$, $W^+(E_8)$ and $W(E_8)$ have self-normalising Sylow $2$-subgroups.
%
%We have the following isomorphisms:
%
%$$\begin{array}{c|c|c|c|c}
%G_2 & F_4 & E_6 & E_7 & E_8\\
%\hline
%Dih_{12} & GO_4^+(3) & {\rm Aut}(PSp_4(3)) &  \mathbb{Z}/2\mathbb{Z}\times PSp_6(2) & 2.PS\Omega_8^+(2).2\\
%\end{array}
%$$
%
%The result follows clearly for $G_2$ and from Carter-Fong for $F_4$.
%The remaining 3 cases can be checke in GAP {\color{red} would like to not use GAP but have checked all 3}
\end{proof}
\end{lm}

\subsection{Normalisers of Sylow $2$-subgroups}

Let $H$ be a finite group and $Q$ be a Sylow $2$-subgroup of $H$. In this section we consider when $\N_H(Q)=\C_H(Q)Q$, for $H$ a group of Lie type. The following remark will be helpful in answering this question.

\begin{rem}\label{central subgroup}
	
	Let $H$ and $Q$ be as above.
	By Schur-Zassenhaus, we have $\N_H(Q)=Q\rtimes K$ for some subgroup $K$ of $\N_H(Q)$.	In particular, $\N_H(Q)=\C_H(Q)Q$ if and only if $K\lhd \N_H(Q)$.
	
	For any central subgroup $Z\leq \Z(H)$ let $\ol {U}$ denote the image of any subgroup $U$ of $H$ in the quotient $H/Z$.
	Observe that $K$ is the unique Hall $2'$-subgroup of $KZ$ and thus a characteristic subgroup of $KZ$. In particular, $\N_H(K)=\N_H(KZ)$ and similarly $\N_H(Q)=\N_H(QZ)$.
	Thus $\ol{K}$ is a complement to $\ol{Q}$ in $\N_{\ol{H}}(\ol{Q})=\N_H(QZ)/Z$.
	As $KZ\lhd \N_H(Q)$ if and only if $\ol{K}\lhd\N_{\ol{H}}(\ol{Q})$, it follows that $\N_H(Q)=\C_H(Q)Q$ if and only if $\N_{\ol{H}}(\ol{Q})=\C_{\ol{H}}(\ol{Q})\ol{Q}$.
\end{rem}

We use the following theorem by Malle \cite[Theorem 5.19]{MalleHZ} which is based on work by Aschbacher.

\begin{thm}\label{Malle}
	 Let $\bH$ be a simple algebraic group and $F:\bH \to \bH$ a Frobenius endomorphism defining an $\mathbb{F}_q$-structure on $\bH$. Let $d$ be the order of $q$ modulo $4$ and $\mathbf{S}$ a Sylow $d$-torus of $(\bH,F)$. Assume that $\bH^F$ is not isomorphic to $\Sp_{2n}(q)$ with $n \geq 1$ and $q \equiv 3,5 $ $\mathrm{mod}$  $8$. Then there exists a Sylow $2$-subgroup $Q$ of $\bH^F$ with $\N_{\bH^F}(Q) \leq \N_{\bH^F}(\mathbf{S})$.
\end{thm}

We can now answer the question posed at the beginning of this section. Note that a similar result to the following corollary was obtained in \cite[Theorem 1]{Kon05}.

\begin{cor}\label{star property}
	Keep the assumption of Theorem \ref{Malle} and let $Q$ be a Sylow $2$-subgroup of $H:=\bH^F$. Then $\N_H(Q)=\C_H(Q) Q$.
	Moreover, for $\bH \hookrightarrow \tilde{\bH}$ a regular embedding and $\wt{Q}$ a Sylow $2$-subgroup of $\tilde{H}:=\tilde{\bH}^F$ with $Q=\wt{Q}\cap H$, then $\N_{\tilde{H}}(Q)=\N_{\tilde{H}}(\tilde{Q})=\C_{\tilde{H}}(\tilde{Q}) \tilde{Q}$.
\end{cor}

\begin{proof}
	As in Remark \ref{central subgroup}, take $K$ a complement to $Q$ in $\N_H(Q)$.
	According to Theorem \ref{Malle}, $K\leq \N_H(\bT)$, where $\bT=\C_{\bG}(\bS)$ is a maximal torus of $\bH$, see \cite[Lemma 3.17]{Cabanesgroup}. In particular, $K$ normalises $Q\bT^F$.
	As $\mathbf{S}$ is $d$-split with $d \in \{1,2\}$, the group $W$, where $W:=\N_{\mathbf{H}^F}({\bf T})/{\mathbf{T}^F}$, is again isomorphic to a Weyl group (use \cite[page 121]{MarcBook} and \cite[Corollary B.23]{MT}).
	Hence $Q\bT^F/\bT^F$, which is a Sylow $2$-subgroup of $W$, is self-normalising in $W$ by Lemma \ref{self normalising}.
	Thus $K\leq Q\bT^F=\bT^F_{2'}\rtimes Q$.
	As $K$ is a $2'$-group, then $K\leq \bT^F_{2'}$ and so $[K,Q]\leq Q\cap \bT^F_{2'}=1$.
	In other words $K\leq \C_H(Q)$.
	This proves the first statement.
	
	Next observe that $\tilde{\bH}^F / \Z(\tilde{\bH})^F \cong \bH_{\mathrm{ad}}^F$ and the assumption of Theorem \ref{Malle} is always satisfied for $\bH_{\mathrm{ad}}^F$. 
	Thus by applying Remark~\ref{central subgroup} it follows that $\N_{\tilde{H}}(\tilde{Q})=\C_{\tilde{H}}(\tilde{Q}) \tilde{Q}$. 
	Therefore it remains to show that $\N_{\tilde{H}}(Q)=\N_{\tilde{H}}(\tilde{Q})$.
	As any two Sylow $2$-subgroups above $Q$ must be conjugate by an element of $\N_{\wt{H}}(Q)$, it suffices to consider a fixed $\wt{Q}\in{\rm Syl}_2(\wt{H})$ lying above $Q$.

	For groups of type $A$ this follows from \cite[Theorem 1]{Kon05}. 
	In the remaining cases $\tilde{H}/H\Z(\tilde{H})$ is either a $2$- or a $2'$-group. Note that if $\wt{H}/H\Z(\wt{H})$ is a $2'$-group, then $\wt{Q}=Q\Z(\wt{H})_2$ is the unique Sylow $2$-subgroup of $\tilde{H}$ containing $Q$ and so $\N_{\wt{H}}(\wt{Q})=\N_{\wt{H}}(Q)$. Thus assume that $\wt{H}/H\Z(\wt{H})$ is a $2$-group.
	For $\wt{\bf T}:=\bf T \Z(\tilde{\bG})$ a maximal torus of $\wt{\bf H}$, we have $\wt{Q}:=\wt{\bT}^F_{2}Q$ is a Sylow $2$-subgroup of $\wt{H}=H\Z(\wt{H})\wt{Q}$ and $[K,\wt{\bT}^F]=1$.
	Thus $\N_{\wt{H}}(Q)=\N_H(Q)\Z(\wt{H})\wt{Q}=K\Z(\wt{H})\wt{Q}\leq \C_{\wt{H}}(\wt{Q})\wt{Q}$.
\end{proof}

\subsection{Groups of Lie type}\label{setup}

The following section is used to introduce the setup which will be in place for the remainder of this article.
Let $\bG$ be a simple algebraic group of simply connected type defined over an algebraic closure of $\mathbb{F}_p$ for some odd prime $p$. We adopt the notation of \cite[Section 2.B]{MS}. In particular, $F_0: \bG \to \bG$ denotes a field endomorphism inducing an $\mathbb{F}_p$-structure on $\bG$ and for every symmetry of the Dynkin diagram associated to $\bG$ we have a graph automorphism $\gamma: \bG \to \bG$. We consider a Frobenius endomorphism $F:=F_0^m \gamma$ with $\gamma$ a (possibly trivial) graph automorphism of $\bG$ such that $F$ defines an $\mathbb{F}_q$-structure on $\bG$, where $q=p^m$. In addition, we let $\bG \hookrightarrow \tilde{\bG}$ be the regular embedding constructed in \cite[Section 2.B]{MS}.

 We will also assume until Section \ref{rest} that $\bG^F$ is not of type $C_n(q)$, $n \geq 1$, or ${}^3 D_4(q)$ whenever $q \not\equiv  1 \, \mathrm{mod}\, 8$.

%We recall the notation of Malle--Späth \cite{MS}.
Denote by $d$ the order of $q$ modulo $4$. We let $\bT$ be a maximally split torus of $\bG$ with corresponding Weyl group $\mathbf{W}$. We set $\bV:=\langle n_\alpha(1) \mid \alpha \in \Phi\rangle \subset \N_\bG(\bT)$, and $\bH:= \bV \cap \bT$. We define $v:=1$ if $d=1$ and $v:= \tilde{w_0}$ if $d=2$, where $\tilde{w_0}$ is the canonical representative in $\bV$ of the longest element $w_0 \in \mathbf{W}$ as defined in \cite[Section 3.A]{MS}. We recall \cite[Notation 3.3]{MS}:

\begin{notation}   \label{not:3.4}
	As before let $F:=F_0^m \gamma$ be a fixed Frobenius endomorphism of $\bG$. Let $E_1$ be the subgroup of $\Aut(\bG)$ generated
	by the graph automorphisms which commute with $\gamma$. Set $e:=\mathrm{o}(\gamma) \mathrm{exp}(E_1) \mathrm{o}(v)$. Let $E:= \C_{em} \times E_1$ act on
	${\tilde \bG}^{F_0^{2em}}$ such that the first summand $\C_{2em}$ of $E$ acts
	by $\langle F_0 \rangle$ and the second by the group generated by graph automorphisms.
	Note that this action is faithful. Let $\wh F_0, \wh \gamma, \wh F \in E$ be
	the elements that act on ${\tilde \bG}^{F^{2em}_0}$ by $F_0$, $\gamma$ and $F$,
	respectively.
\end{notation}

\begin{lm}\label{Sylow}
	%	Suppose that $\bG$ is not of type $A$.
	The torus $\bT$ contains a Sylow $d$-torus $\bS$ of $(\bG,vF)$. Moreover, $\bT=\C_{\bG}(\bS)$ and $N=T V$, where $N:=\N_\bG(\bS)^{vF}$, $T:=\bT^{vF}$ and $V:=\bV^{vF}$.
\end{lm}

\begin{proof}
	See \cite[Lemma 3.2]{MS} and \cite[Section 5.1]{Equivariantcharacter}.
\end{proof}

	Note that $E$ stabilises $N$, $T$, $V$ and hence $H:=\bH^{vF}$.
In what follows both the groups $\bG^F$ and $\bG^{vF}$ will be considered. 
Therefore, in addition to the notation in Malle--Späth \cite{MS} the objects from $G_0:=\bG^F$ will be denoted with a subscript $0$, e.g. $T_0:=\bT^{F}$, $N_0:=\N_{G_0}(\bS)$ and $W_0:=\bW^{F}$. The following lemma provides a tool to pass between the groups $G_0=\bG^{F}$ and $G:=\bG^{vF}$ and compare them:

\begin{lm}\label{TwistIsom}
	Let $g \in \bG$ such that $gF(g)^{-1}=v$. Then the map
	$$\iota: \tilde\bG^{F_q^{2e}} \rtimes E \to \tilde\bG^{F_q^{2e}} \rtimes E, \, x \mapsto x^{g^{-1}}$$
	is an isomorphism which maps $\bG^F \rtimes E$ onto $\bG^{vF} \rtimes E$.
\end{lm}

\begin{proof}
	See the proof of \cite[Proposition 5.3]{Equivariantcharacter}. 
\end{proof}

Since the image of $\hat{F}$ under $\iota$ is $v \hat{F}$ we obtain an isomorphism $(\bG^F \rtimes E)/\langle \hat F \rangle \cong (\bG^{vF} \rtimes E)/ \langle v \hat F \rangle$. From Theorem \ref{Malle} we are now able to explicitly construct a Sylow $2$-subgroup of $G:=\bG^{vF}$. First, we let $T_2$ and $V_2$ be a Sylow $2$-subgroup of $T$ and $V$ respectively. We define $P:=T_2 V_2$ which forms a Sylow $2$-subgroup of $G$ and conclude that $Q:=\iota^{-1}(P)$ is a Sylow $2$-subgroup of $G_0$. In the next section, we show that $P$ can be chosen to be $E$-stable.

\subsection{Automorphisms}

\begin{lm}\label{longest element}
%	\begin{enumerate}[label=(\alph*)]
%		\item 
Let $W$ be a Weyl group of irreducible type.
If $W$ is of type $A_n$ ($n \geq 2$), $D_n$ ($n$ odd) or $E_6$, then the longest element $w_0 \in W$ acts as the (unique) non-trivial graph automorphism of order $2$ on $W$. In the remaining cases, $w_0 \in \Z(W)$.
%	\end{enumerate}
\end{lm}

\begin{proof}
	Follows from remarks following \cite[Corollary B.23]{MT}.
\end{proof}

For $\bW$ and $\bV$ as in Section~\ref{setup}, it is an obvious question whether the action of the representative $\tilde{w}_0$ of $w_0$ in $\bV$ can be described in a similar way. The next lemma gives a positive answer to this.

\begin{lm}
	Whenever $w_0 \in \Z(\bW)$ then we have $\tilde{w}_0 \in Z(\bV)$. In the remaining cases we have $\mathrm{C}_{\bV}(\tilde{w}_0)=\mathrm{C}_{\bV}(\gamma_0)$, where $\gamma_0$ is the graph automorphism which acts as $w_0$ on $\bW$. 
\end{lm}

\begin{proof}
	This follows from the citations given in the proof of \cite[Lemma 3.2]{MS}.
\end{proof}

\begin{lm}\label{E-stable Sylow}
	There exists an $E$-stable Sylow $2$-subgroup $W_2$ of $\bW^{w_0 F}$ with $w_0 \in Z(W_2)$. Moreover, $W_2$ is a Sylow $2$-subgroup of $\bW^F$.
\end{lm}

\begin{proof}
	Let us first assume that $\bW$ is not of type $D_{2n}$. Using the formulas given on the bottom of \cite[page 121]{MarcBook} together with the well-known order formulas for Weyl groups, we deduce that  $|\bW:\bW^\sigma|$ is odd for any graph automorphism $\sigma$.
	%W({}^2 E_6)=W(F_4) and |W(E_6)|=518410, |W(F_4)|=1152
	Moreover, $w_0$ is $\sigma$-stable so we can choose $W_2$ to be a Sylow $2$-subgroup of $\bW^\sigma$  with $w_0 \in \Z(W_2)$ by Lemma \ref{longest element}.
	
In type $D_{2n}$ with $2n>4$, the element $w_0$ corresponds to a central element of $\bW$ and so $\bW^{w_0 F}=\bW^F$.
If $2n>4$, it can be assumed that $F$ is a field automorphism, otherwise, $E$ acts trivially on $\bW^{F}$.
It therefore suffices to find a $\sigma$-stable Sylow $2$-subgroup of $\bW$ for $\sigma$ the graph automorphism. However $\sigma$ has order $2$, $\bW$ has an odd number of Sylow $2$-subgroups and so by the orbit-stabiliser theorem one must be fixed by $\sigma$.

This leaves the case when $\bW$ is of type $D_4$.
As before, it can be assumed $F$ is a field automorphism, otherwise the group $E$ acts trivially on $\bW^F$. In this case $\bW^F=\bW$ and it was shown in the proof of Lemma~\ref{self normalising} that $W(D_4)$ has a Sylow $2$-subgroup which is $E$-stable.
\end{proof}

Let $V_2$ be the preimage of the Sylow $2$-subgroup $W_2$ from Lemma \ref{E-stable Sylow} under the natural projection map $V \to W$.

\begin{cor}\label{construction}
	The Sylow $2$-subgroup $P:=T_2V_2$ of $G$ is $E$-stable.
\end{cor}

\begin{proof}
 The group $\bH$ is a characteristic subgroup of $\bV$ and so $H \subseteq V_2$. Since $V/H \cong W$ and the image of $V_2$ in $W$ is $E$-stable it follows that $V_2$ is $E$-stable.
\end{proof}

As a consequence of this the Sylow $2$-subgroup $Q=\iota^{-1}(P)$ of $G_0$ is $D$-stable, where $D:=\iota^{-1}(E)/ \langle \hat F \rangle$.

\section{Parametrisations of characters}

\subsection{Duality and character bijections of tori}\label{DualityChar}

We show how duality can be used to provide bijections between certain characters of tori. For $(\bG,\bT,F)$ from Section~\ref{setup}
take $(\bG^\ast,\bT^\ast,F^\ast)$ to be a triple in duality as in \cite[Definition 13.10]{DM}.
%We do not know any other way of obtaining these bijections.
Denote by $W_2^\ast$ and $w_0^\ast$ the image of $W_2$ respectively $w_0$ under the isomorphism $\bW \to \bW^\ast$ induced by duality. 
In the following we let $v^\ast$ be a fixed preimage in $\N_{\bG^\ast}(\bT^\ast)$ of $w_0^\ast$ whenever $d=2$, otherwise $v^\ast:=1$. 
Moreover, we will denote the images of $v$ and $v^\ast$ in $\bW$ respectively $\bW^\ast$ by the same symbol.

\begin{prop}\label{conversion}
Let $W_2$ be as in Lemma~\ref{E-stable Sylow}.
Then there exists a bijection $$\alpha:\Irr(\bT^F)^{W_2} \to \Irr(\bT^{vF})^{W_2}.$$ Moreover, if $\sigma: \bG \to \bG$ is a bijective morphism with $\sigma(\bT)=\bT$ commuting with $F$ such that $\sigma(v)=v$, then this bijection is equivariant with respect to $\sigma$.
\end{prop}

\begin{proof}
By duality we obtain a bijection $\Irr(\bT^F) \to (\bT^\ast)^{F^\ast}$. Let $\sigma$ be a bijective morphism of $\bG$ which stabilises $\bT$. Then there exists a unique bijective morphism (up to $(\bT^\ast)^{F^\ast}$-conjugation) $\sigma^\ast: \bG^\ast \to \bG^\ast$ commuting with $F^\ast$ and in duality with $\sigma$ such that this bijection is $(\sigma,\sigma^\ast)$-equivariant. 
Then we obtain a bijection $\beta_0:\Irr(\bT^F)^{W_2} \to ((\bT^\ast)^{F^ \ast})^{W_2^\ast}$.

The triple $(\bG,\bT,vF)$ is in duality with $(\bG^\ast,\bT^\ast,F^\ast v^\ast)$.
Thus we similarly obtain a $(\sigma,\sigma^\ast)$-equivariant bijection $\Irr(\bT^{vF}) \to (\bT^\ast)^{F^\ast v^\ast}$.
Furthermore, since $W_2 \subset \C_{W}(w_0) $, this induces a bijection $\beta:\Irr(\bT^{v F})^{W_2} \to ((\bT^\ast)^{F^\ast v^\ast})^{W_2^\ast}$. 
However $v^\ast \in W_2^\ast$ and so $((\bT^\ast)^{F^\ast})^{W_2^\ast}=((\bT^\ast)^{F^\ast v^\ast})^{W_2^\ast}$. In particular, we obtain a bijection $$\alpha:=\beta^{-1} \circ \beta_0: \Irr(\bT^F)^{W_2} \to \Irr(\bT^{vF})^{W_2}$$ which is $\sigma$-equivariant as both $\beta$ and $\beta_0$ are $(\sigma,\sigma^\ast)$-equivariant.
\end{proof}

%In particular, the map in the previous proposition is $C_{W^F}(v)$-equivariant.

\begin{rem}
%Characters of the quotient $\tilde \bT^F/ \bT^F$ can be constructed by using the isomorphism $Z(\tilde \bG) \cong \tilde \bT / \bT$ and by taking the $F$-fixed points of this construction. Note that clearly $Z(\tilde \bG)^F = Z(\tilde \bG)^{vF}$ which means that the characters of $\tilde \bT^{vF}/ \tilde \bT^{vF}$ and $\tilde \bT^{F}/ \tilde \bT^{F}$ are in bijection. More 
By \cite[Equation (15.2)]{MarcBook} duality induces bijections $ \mathrm{Z}(\tilde \bG^\ast)^{F^\ast} \to \Irr(\tilde \bT^F / \bT^F)$ and $\mathrm{Z}(\tilde \bG^\ast)^{ F^\ast v^\ast} \to \Irr(\tilde \bT^{vF} / \bT^{vF})$. In particular, if $\theta_0 \in \Irr(\tilde \bT^F / \bT^F)$ is the character corresponding to $z \in \mathrm{Z}(\tilde \bG^\ast)^{F^\ast}$ then $\theta:=\theta_0 \circ \iota \in \Irr(\tilde \bT^{vF} / \bT^{vF})$ is the character corresponding to the same central element $z \in \mathrm{Z}(\tilde \bG^\ast)^{F^\ast}$. Thus, we will denote the character $\theta$ and $\theta_0$ by the same symbol $\hat{z}$. 
\end{rem}

In the following we will employ the notation introduced in Section \ref{setup} with respect to the dual group $\bG^\ast$.
%In particular, as in the proof of Proposition~\ref{conversion}, $T^\ast:=(\bT^\ast)^{F^\ast v^\ast}$ corresponding to $T:=\bT^{vF}$ under duality.
Moreover, for $s \in \bT^\ast$ we denote by $\bW^\circ(s)$ the Weyl group of $\C_{\bG^ \ast}^\circ(s)$ with respect to the maximal torus $\bT^\ast$ and $\bW(s):=\C_{\bW^\ast}(s)$.

\begin{prop}\label{connected Weyl}
	For $s \in (T_2^\ast)^{W_2^\ast}$ we have $v^\ast \in W^\circ(s)$.
\end{prop}

\begin{proof}
It can be assumed that $q \equiv 3 \, \mathrm{mod} \, 4$, otherwise $v^\ast$ is trivial.
In particular, $v^\ast$ is our fixed preimage of $w_0^\ast$ in $\N_{\bG^\ast}(\bT^\ast)$ and $s$ centralises a Sylow $2$-subgroup of $G^\ast$.

We have $w_0^\ast \in W(s)$, so $w_0^\ast \in W^\circ(s)$ whenever $\bC:=\C_{\bG^\ast}(s)$ is connected. We can therefore assume that $\bC$ is disconnected. Let us first suppose that $\bG$ is not of type $A_n$. By the proof of \cite[Theorem 8.7]{MS}, using that $s$ centralises a Sylow $2$-subgroup of $G^ \ast$, the centraliser $\bC^\circ$ contains a maximally split torus $\mathbf{S}$ of $(\bG,F^\ast v^\ast)$. As $\bT^\ast \subset \bC^\circ$ there exists $x \in \bC^\circ$ such that $\bS={}^x \bT^\ast$. Let $h \in \bG^\ast$ such that $F^\ast(v^\ast)=h F^\ast(h^{-1})$. In particular, ${}^{h^{-1}} \bS$ is a maximally split torus of $(\bG^\ast,F^\ast)$. Assume first that $F$ is untwisted, i.e. $F^\ast$ induces the identity on $\bW^\ast$. Since $\bT^\ast$ is also a maximal $1$-split torus of $(\bG^\ast,F^\ast)$ we have $(h^{-1} x)^{-1} F^\ast(h^{-1} x)=x^{-1} h^{-1} F^\ast(h) F^\ast(x) \in \bT^\ast$, see \cite[Application 3.23]{DM}. Since $x \in \bC^\circ$ and the image of $h F^\ast(h^{-1})$ in $\bW^\ast$ is $w_0^\ast$, we find that $w_0^\ast\in W^\circ(s)$.
	
Assume now that $F$ is twisted, i.e. $\phi:= F^\ast v^\ast$ induces the identity on $\bW^\ast$. Here, we use that both $\bS$ and ${}^h \bT^\ast$ are maximally split tori of $(\bG^\ast,\phi)$. In particular, $\phi$ acts trivially on the Weyl group $\bW({}^h \bT^\ast)$ and again by \cite[Application 3.23]{DM} we have $(xh^{-1})^{-1} \phi(xh^{-1}) \in {}^h \bT^\ast$. This yields $x^{-1} \phi(x) \phi(h^{-1}) h \in \bT^\ast$. As $\phi(h^{-1}) h=F^\ast(v_0^\ast)$ we again deduce that $w_0^\ast \in W^\circ(s)$.
	
	Finally, if $\bG^F$ is of type $A_n(\varepsilon q)$, $n >1$, we use the proof of \cite[Theorem 3.3]{MalleGalois}. As $s$ centralises a Sylow $2$-subgroup of $G^\ast$ it follows by the arguments given there (together with the information in \cite[Table 4.5.1]{GLS}) that $n+1$ is necessarily a power of $2$ and $\bC$ is of rational type $A^2_{\frac{n-1}{2}}(\varepsilon q).2$ or $A_{\frac{n-1}{2}}(q^2).2$. A calculation shows that $\bC$ can only contain a Sylow $2$-subgroup of $\bG^{F^\ast v^\ast}$ when $\bC$ has rational type $A^2_{\frac{n-1}{2}}(q).2$. In this case $\bC$ contains a maximal $1$-split torus of $(\bG^\ast,F^\ast v^\ast)$ and the arguments from before apply also here.
\end{proof}

The previous proposition provides a way to compare the characters of $\tilde{\bT}^F$ lying over a $W_2$-stable character of $\bT^F$ with the analogous situation arising from $\tilde{\bT}^{vF}$. The following result will be used in Section \ref{tilde G}.

\begin{prop}\label{conversion2}
Let $\alpha$ be the bijection as in Proposition~\ref{conversion}.
Then there exists a bijection $$\tilde \alpha:\Irr(\tilde \bT^F \mid \Irr(\bT^F)_2^{W_2}) \to \Irr(\tilde \bT^{vF} \mid \Irr(\bT^{vF})_2^{W_2} )$$ such that $\alpha \circ \mathrm{Res}_{\bT^F}^{\tilde \bT^F}=\mathrm{Res}_{\bT^{vF}}^{\tilde \bT^{vF}} \circ \tilde \alpha$ and if $\hat z \in \Irr(\tilde \bT^F/ \bT^F)$ then  $\tilde{\alpha}(\hat z)=\hat z$.

Additionally let $\sigma: \tilde{\bG} \to \tilde{\bG}$ be a bijective morphism commuting with $F$ such that $\sigma|_{\bG}$ is as in Lemma \ref{conversion}.
If $\tilde \lambda \in \Irr(\tilde \bT^F \mid \Irr(\bT^F)_2^{W_2})$ is such that ${}^{\sigma}\tilde\lambda=\tilde \lambda \hat{z}$ for some $z \in \mathrm{Z}(\tilde{\bG}^\ast)^{F^\ast}$, then we have ${}^{\sigma} \tilde \alpha(\tilde \lambda)=\tilde \alpha(\tilde \lambda) \hat{z}$.

\end{prop}

\begin{proof}
Duality yields again a bijection $\Irr(\tilde{\bT}^F) \to (\tilde{\bT}^\ast)^{F^\ast}$. Let $\tilde{s} \in (\tilde{\bT}^\ast)^{F^\ast}$ be a semisimple element corresponding to a character $\tilde{\lambda} \in \Irr(\tilde \bT^F \mid \Irr(\bT^F)^{W_2})$ under this bijection. The map $i: \bG \to \tilde{\bG}$ induces by duality a surjective map $i^\ast:\tilde{\bG}^\ast \to \bG^\ast$ and the image $s:=i^\ast(\tilde{s})$ of $\tilde{s}$ lies in $((\bT^\ast)^{F^\ast})^{W_2^\ast}=((\bT^\ast)^{F^\ast v^\ast})^{W_2^\ast}$. In particular, by Proposition \ref{connected Weyl} we have $v^\ast \in W^\circ(s)$. The map $\iota^\ast$ yields an isomorphism $W^\circ(\tilde{s}) \cong W^\circ(s)$ and so we deduce that ${}^{v^\ast} \tilde{s}=\tilde{s}$. In particular, $\tilde{s}$ is $F^\ast v^\ast$-stable. Let $\tilde{\alpha}(\tilde{\lambda}) \in \Irr(\tilde{\bT}^{vF})$ denote the character corresponding to $\tilde{s}$ under the bijection $\Irr(\tilde{\bT}^{v F}) \to (\tilde{\bT}^\ast)^{F^\ast v^\ast}$. One then checks easily that the so-obtained map $$\tilde \alpha:\Irr(\tilde \bT^F \mid \Irr(\bT^F)^{W_2}) \to \Irr(\tilde \bT^{vF} \mid \Irr(\bT^{vF})^{W_2} )$$ is a well-defined bijection which has all the required properties. 
\end{proof}

\subsection{Local characters}

Recall that $P$ denotes the Sylow $2$-subgroup of $G=\bG^{vF}$ constructed in Corollary \ref{construction} and $Q=\iota^{-1}(P)$ its preimage under $\iota$, which is a Sylow $2$-subgroup of $G_0=\bG^{F}$.
In this section we make use of the explicit description of $P$ to provide a description of the odd degree characters in the principal $2$-block of $\N_{G_0}(Q)$.
For a finite group $H$ we denote by $\Irr_{2'}(H)$ its set of irreducible characters of odd degree.

\begin{prop}\label{direct product}
For $P=T_2V_2$ as in Corollary \ref{construction}, there is a bijection
		$$\Irr_{2'}(P) \to \Irr(T_2)^{W_2} \times \Irr_{2'}(W_2).$$
\end{prop} 

\begin{proof}
	Any character of $\Irr_{2'}(P)$ (that is any linear character of $P$) covers a $P$-invariant character of the normal subgroup $T_2$ of $P$. Since $P/T_2 \cong V_2/H \cong W_2$ the statement follows from \cite[Corollary 3.13]{MS} and Gallagher's theorem.
\end{proof}

We conjecture that the result in part (b) holds in general.
Recall as in Section~\ref{DualityChar} that $W_2 ^\ast$ denotes the image of $W_2$  and $T^\ast:=(\bT^\ast)^{F^\ast v^\ast}$ corresponding to $T:=(\bT)^{vF}$ under duality.

\begin{prop}\label{local}
Let $B$ be the principal $2$-block of $\N_{G_0}(Q)$.
Then for $Z:=(T_2^\ast)^{W_2^\ast}$, there is a bijection $\Irr_0(B) \to Z \times \Irr_{2'}(W_2)$.
\end{prop}

\begin{proof}
By Corollary \ref{star property} we have $\N_{G_0}(Q)=\C_{G_0}(Q) Q$ and thus by \cite[Theorem 9.12]{NavarroBook} restriction defines a bijection $\Irr_0(B) \to \Irr_{2'}(Q)$. 
As in the proof of Proposition~\ref{conversion}, duality  provides a bijection $\Irr(T)^{W_2} \to (T^\ast)^{W_2^\ast}$, which yields a bijection
	$$\Irr(T_2)^{W_2} \to (T_2^\ast)^{W_2^\ast}.$$
The result thus follows from Proposition \ref{direct product} using that $P \cong Q$.
\end{proof}

\begin{rem}
	Let $P^\ast$ be a Sylow $2$-subgroup of $G^\ast$. As for $G$, it can be obtained as an extension of $T_2^\ast$ by $W_2^\ast$. Therefore, $Z:=(T_2^\ast)^{W_2^\ast}$ is a central subgroup of $P^\ast$. We believe that $Z$ should coincide in most cases with $\Z(P^\ast)$. For instance if $G^\ast$ is of type $A$ then this is the case by \cite[Lemma 13.17(ii)]{MarcBook}.
\end{rem}

\subsection{Global characters}\label{GlobalChar}

This section focuses on the height zero characters of the principal block for $G_0=\bG^F$ as in Section~\ref{setup}.
First we count these characters by counting those in $\bG^{vF}$ using Malle's parametrisation of $2'$-degree characters.

\begin{lm}\label{GlobalCount}
The principal $2$-block of $G=\bG^{vF}$ contains $|Z| \times |\Irr_{2'}(W_2)|$ height zero characters, where $W_2$ and $Z$ are taken from Lemma \ref{E-stable Sylow} and Proposition~\ref{local} respectively.
\begin{proof}
The odd degree characters of $G$ have been parametrised by Malle \cite[Proposition 7.3]{MalleHZ}. 
However, by the proof of \cite[Theorem A]{Enguehard}, the principal $2$-block is the unique unipotent block of maximal defect.
Therefore using Malle's explicit parametrisation, it follows that the height zero characters of the principal block of $G$ are in bijection with pairs $(s,\phi)$, where $s \in Z$ and $\phi \in \Irr_{2'}(W(s))$, where $W(s):=\C_W(s)$.
As $s \in Z$, then $W_2\leq \C_W(s)$ and thus by the main result of \cite{MS} we have a McKay-bijection $\Irr_{2'}(W(s)) \to \Irr_{2'}(W_2)$. 
\end{proof}
\end{lm}

\begin{cor}
	Recall that $\bG$ is simple of simply connected and $F$ is a Frobenius map with $\bG^ F \not\cong \{\Sp_{2n}(q),{}^3 D_4(q) \}$ whenever $q \not \equiv 1 \, \mathrm{mod} \, 8$.
	Then the Alperin--McKay conjecture holds for the principal $2$-block of $\bG^F$.
\end{cor}

\begin{proof}
	This follows from Proposition \ref{local} and Theorem \ref{GlobalCount}.
\end{proof}

Define
	$$\mathcal{P}_0:= \{(\lambda_0,\eta_0) \mid \lambda_0 \in \Irr(T_0) \text{ and } \eta_0 \in \Irr_{2'}(W_0(\lambda_0)) \},$$
where $W_0(\lambda_0) :=(N_0)_{\lambda_0}/T_0$.
From the proof of \cite[Theorem 6.3]{MS} there is a surjective map onto the principal Harish-Chandra series
	$$\begin{array}{clcl}
           \Pi_0: & \mathcal{P}_0 &  \to & \displaystyle\bigcup_{\lambda_0 \in \Irr(T_0) } \mathcal{E}(G_0,(T_0,\lambda_0))\\
 &  (\lambda_0,\eta_0)  & \mapsto & R_{T_0}^{G_0}(\lambda_0)_{\eta_0}\\
           \end{array}$$
which becomes injective on $W_0$-orbits. 

The main aim is to find a suitable subset of $\mathcal{P}_0$ to parametrise the height zero characters of the principal block $b$ of $G_0=\bG^F$.
If $R_{T_0}^{G_0}(\lambda_0)_{\eta_0}$ has $2'$-degree, then  by \cite[Lemma 8.9]{MS} it follows that $2 \nmid |W_0:W_0(\lambda_0)|$.
In other words, $W_0(\lambda_0)$ contains a Sylow $2$-subgroup of $W_0$. 
Furthermore, the principal block of $G_0$ is a subset of $\mathcal{E}_{2'}(G_0,1)$, see \cite[Theorem 9.12(a)]{MarcBook}.
However for $s_0 \in T_0^\ast$ in duality with $\lambda_0 \in \Irr(T_0)$, it follows that $s_0$ has $2$-power order if and only if $\lambda_0$ has $2$-power order.
Therefore if $\chi\in \Irr(b)$ lies in $\mathcal{E}(G_0,(T_0,\lambda_0))$, then $\lambda_0$ must have $2$-power order.
Via the decomposition $T_0=(T_0)_2 \times (T_0)_{2'}$, the $2$-power order characters coincide with the set $\Irr((T_0)_2)$, which can be viewed as the characters of $T_0$ with $(T_0)_{2'}$ in their kernel.
Thus for $W_2$ the fixed Sylow $2$-subgroup of $W$ from  Lemma \ref{E-stable Sylow} define 
$$(\mathcal{P}_0)_2:= \{(\lambda_0,\eta_0)\in  \mathcal{P}_0 \mid \lambda_0 \in \Irr((T_0)_2)^{W_2} \}$$
and set $\Pi_{\mathrm{glo}}$ to be the restriction of $\Pi_0$ to $(\mathcal{P}_0)_2$.

\begin{thm}\label{global}
Let $b$ be the principal $2$-block of $G_0$.
Then the map $\Pi_{\mathrm{glo}}$ yields a bijection
	$$\Pi_{\mathrm{glo}}:(\mathcal{P}_0)_2\to \Irr_0(b).$$	
\end{thm}

\begin{proof}

Every character of $\Irr_0(b)$ lies in the principal Harish-Chandra series by \cite[Theorem 3.3]{MalleGalois}. 
That is $\Irr_0(b)\subset \Pi_0(\mathcal{P}_0)$.
If $\chi=R_{T_0}^{G_0}(\lambda_0)_{\eta_0}\in \Irr_0(b)$, then as in the paragraph above, it follows that there is some $W_0$-conjugate $(\lambda_0',\eta_0')$ of $(\lambda_0,\eta_0)$ with $\chi=R_{T_0}^{G_0}(\lambda_0')_{\eta_0'}$ and $\lambda_0'\in  \Irr((T_0)_2)^{W_2}$; in other words $(\lambda_0',\eta_0')\in (\mathcal{P}_0)_2$.
Moreover, as $W_2$ is self-normalising in $W_0$ (Lemma \ref{self normalising}), it follows that  $\lambda_0'$ must be the unique character in its $W_0$-orbit with $W_2\subseteq W_0(\lambda_0')$. 
Hence $\Irr_0(b)\subseteq \Pi_{\mathrm{glo}}((\mathcal{P}_0)_2)$ and each $\chi  \in \Irr_0(b)$ has a unique preimage in $(\mathcal{P}_0)_2$ under $\Pi_{\mathrm{glo}}$.

It remains to show that $\Pi_{\mathrm{glo}}$ is indeed a bijection as stated in the theorem.
By Lemma~\ref{GlobalCount},  it suffices to show that $|(\mathcal{P}_0)_2|=|Z|\times |\Irr_{2'}(W_2)|$.
From the proof of Proposition~\ref{conversion}, there is a bijection $$\Irr((T_0)_2)^{W_2}\to (T_0^\ast)_2^{W_2^\ast}= (T_2^\ast)^{W_2^\ast}=:Z.$$
Furthermore, for each $\lambda_0\in (\mathcal{P}_0)_2$, there is a McKay-bijection $\Irr_{2'}(W_0(\lambda_0)) \to \Irr_{2'}(W_2)$ by the main result of \cite{MS}.
Thus $|(\mathcal{P}_0)_2|=|Z||\Irr_{2'}(W_2)|$.
\end{proof}

\begin{example}\label{example}
	Consider $G=\SL_2(q)$ and assume that $q \equiv 3 \, \mathrm{mod} \, 4$. Recall that this case was excluded in Section \ref{setup}. The principal $2$-block $b$ of $G$ has four height zero characters. There are four characters in the principal $1$-Harish-Chandra series corresponding to characters in $(\mathcal{P}_0)_2$, but only two of them are of $2'$-degree. On the other hand, all four $2'$-degree characters of $b$ lie in the principal $2$-Harish-Chandra series. 
\end{example}

\begin{rem}\label{hz characters}
Assume that $\bG$ is not of type $A_n$, $D_{2n+1}$, $n >1$, or $E_6$ so that the longest element $w_0 \in \bW$ acts by inversion on the torus $\bT$. It follows from the remarks after \cite[Lemma 8.5]{MS} that all $2'$-characters lie in the union of Lusztig series $\mathcal{E}(G_0,s)$ with $s$ of $2$-power order. By the proof of \cite[Theorem A]{Enguehard}, the principal $2$-block is the unique unipotent block of maximal defect. Hence, in these cases the Alperin--McKay conjecture for the principal $2$-block is tantamount to the McKay conjecture for the prime $2$.
\end{rem}

\section{Action of automorphisms}

One of the key steps in the proof of Theorem \ref{global} was the existence of a McKay-bijection $\Irr_{2'}(W_0(\lambda)) \to \Irr_{2'}(W_2)$. We will now construct such a bijection with suitable equivariance properties. For this we need the following lemma, whose proof follows \cite[Lemma 2.1]{notLie}. 

\begin{lm}\label{McKay}
	Let $H$ be a finite group and $A \subset \mathrm{Aut}(H)$ a cyclic group of automorphisms stabilizing the normaliser $M$ of a Sylow $2$-subgroup of $H$. Then there exists an $A$-equivariant McKay bijection $\Irr_{2'}(H) \to \Irr_{2'}(M)$.
\end{lm}

\begin{proof}
According to the main result of \cite{MS} there exists such a McKay bijection. We only need to show that it can be chosen to be $A$-equivariant. For $i \mid r:=|A|$ let $a_i$ (resp. $b_i$) be the number of $\theta \in \Irr_{2'}(H)$ (resp. $\theta \in \Irr_{2'}(M)$) with $|A_{\theta}|=i$. As $A$ is cyclic, it suffices to show that $a_i=b_i$ for all $i \mid r$. Let $q$ be a prime dividing $r$ and set $s:=r/q$. By induction on $r$ we can assume that $a_i=b_i$ for all $i \notin \{r,s\}$ and $$a_{s}+a_r=b_{s}+b_r.$$
Let us first assume that $2 \nmid r$. By Clifford theory we have $|\Irr_{2'}(HA)|=\sum_{i \mid r} a_i i^2/r$
and similarly $|\Irr_{2'}(MA)|=\sum_{i \mid r} b_i i^2/r$. Since the McKay-conjecture holds for $HA$ we have $|\Irr_{2'}(HA)|=|\Irr_{2'}(MA)|$ and so $$a_s s^2/r+a_r r= b_s s^2/r +b_r r.$$
We therefore have two homogeneous linear equations in the variables $a_{s}-b_{s}$ and $a_r -b_r$. As the associated coefficient matrix is invertible we deduce that $a_{s}=b_{s}$ and $a_r=b_r$. Let's now suppose that $r$ is a power of $2$. In that case, we obtain $|\Irr_{2'}(HA)|=a_r r=|\Irr_{2'}(MA)|=b_r r$. We again deduce that $a_{s}=b_{s}$ and $a_r=b_r$. The general case follows now by using the decomposition $A=A_2 \times A_{2'}$ and coprime arguments.
\end{proof}

\begin{rem}
We note that the existence of an automorphism-equivariant McKay-bijection should also follow from a similar statement as \cite[Theorem B]{JEMS}. As we only need the result in the case of a cyclic automorphism group we have decided not to pursue this.
\end{rem}

Take $T_0=\bT^F$ and $N_0=\N_{G_0}(\bS)$ as in Section~\ref{setup}.
For $\tilde \lambda_0 \in \Irr(\tilde T_0)$ denote $W_0(\tilde{\lambda}_0) :=(N_0)_{\tilde\lambda_0}/T_0$.
Note that if $\tilde \lambda_0 \in \Irr(\tilde T_0\mid \lambda_0)$ for some $\lambda_0\in \Irr(T_0)$, then the factor group $W_0(\lambda_0)/W_0(\tilde{\lambda}_0)$ is an abelian group by the proof of \cite[Proposition 3.16]{MS}.

\begin{lm}\label{equivariance}
	Let $\lambda_0 \in \Irr((T_0)_2)^{W_2}$ and let $\tilde{\lambda}_0 \in \Irr((\tilde{T}_0)_2 \mid \lambda_0)$. Then there exist an $E_{\lambda_0}$-equivariant bijection $$f_{\lambda_0}:\Irr_{2'}(W_0(\lambda_0)) \to \Irr_{2'}(W_2)$$ such that $f_{\lambda_0}(\eta_0 \mu_0)=f_{\lambda_0}(\eta_0) \Res^{W_0(\lambda_0)}_{W_2}(\mu_0)$ for every character $\eta_0 \in \Irr_{2'}(W_0(\lambda_0))$ and $\mu_0 \in \Irr(W_0(\lambda_0) / W_0(\tilde \lambda_0))$.
\end{lm}

\begin{proof}
	
	The group 
	$$W_0(\lambda_0)/W_0(\tilde \lambda_0)= \{ w\in W_0 \mid {}^w \tilde \lambda_0 =\tilde \lambda_0 \otimes \nu_0 \text{ some } \nu_0 \in \Irr(\tilde T_0 /T_0)\}/W_0(\tilde \lambda_0)$$
	is always a $2$-group since $\tilde{\lambda}_0$ has $2$-power order and $W_0$ acts trivially on $\tilde{T}_0/T_0$. As $\eta_0$ is a $2'$-character and the quotient is a $2$-group it follows that $\eta_0$ restricts irreducibly to $W_0(\tilde \lambda_0)$. By Gallagher's theorem, the group $\Irr(W_0(\lambda_0)/W_0(\tilde \lambda_0))$ acts fixed point freely on the orbit of $\eta_0 \in \Irr(W_0(\lambda_0))$. On the other hand, every character of $\Irr_{2'}(W_2)$ is linear and thus restricts irreducibly to $W_2(\tilde \lambda_0):=W_2 \cap W(\tilde \lambda_0)$.
	
	Let us first assume that $G_0$ is not of type $D_{2n}(q)$.
	Denote by $E_0$ the stabiliser of $\lambda_0$ in $E$. Observe that $E$ acts by inner automorphisms on $W_0$ and centralises $W_2$ by Lemma \ref{longest element}. In particular, every character of $W_0(\lambda_0)$ and $W_2$ is $E_0$-stable in this case. 	Since the Sylow $2$-subgroup $W_2$ is self-normalising in $W_0$ there exist a McKay bijection $\Irr_{2'}(W_0(\lambda_0)) \to \Irr_{2'}(W_2)$. By the previous discussion it's now easy to construct a bijection $f_{\lambda_0}: \Irr_{2'}(W_0(\lambda_0)) \to \Irr_{2'}(W_2)$ with the required properties.

	Let us now assume that $G_0$ is of type $D_{2n}(q)$.  We use the notation of the proof of \cite[Theorem 3.17]{MS}. Let $\Phi(\tilde\lambda_0)$ be the root system associated to the Weyl group $W_0(\tilde\lambda_0)$. There exists an $E_0$-stable base $\Delta_0$ of $\Phi(\tilde\lambda_0)$. Denote $A_0:= \mathrm{Stab}_{W_0}(\Delta_0)$ which is $E_0$-stable as $\Delta_0$ is. By the proof of \cite[Theorem 3.17]{MS}, $W_0(\lambda_0)=W_0(\tilde \lambda_0) \rtimes A_0$. Moreover, $A_0$ is a $2$-group as already observed above. Let $\eta_0 \in \Irr(W_0(\tilde \lambda_0))$ which extends to a $2'$-character of $W_0(\lambda_0)$ and set $\delta_0:=\mathrm{det}(\eta_0)$. By \cite[Lemma 6.24]{Isaacs} there exists a unique extension $\hat \eta_0 \in \Irr(W_0(\lambda_0))$, such that $\mathrm{det}(\hat\eta_0)=\hat\delta_0$, where $\hat \delta_0$ is the unique extension of $\delta_0$ with $A_0$ in its kernel.
	
	Similarly, we have $W_2=W_2(\tilde \lambda_0) \rtimes A_0$. Thus, any character $\eta_0 \in \Irr(W_2(\tilde \lambda_0))$ covered by a linear character of $W_2$ has a unique extension $\hat \eta_0 \in \Irr(W_2)$ with $A_0$ in its kernel.
	
	Let us now first assume that $G_0$ is not of type $D_4(q)$. Note that $E/\C_E(W_0)$ is cyclic and thus, by Lemma \ref{McKay} there exists an $E_0$-equivariant McKay-bijection $g_{\lambda_0}$ from $\Irr_{2'}(W_0(\lambda_0))$  to $\Irr_{2'}(W_2)$. This induces an $E_0$-equivariant bijection 
$$
\begin{array}{cccc}
f_0: &  \Irr(W_0(\tilde\lambda_0)) \mid \Irr_{2'}(W_0(\lambda_0)) &  \to &  \Irr(W_2(\tilde\lambda_0) \mid \Irr_{2'}(W_2) )\\
 & \eta_0 & \mapsto & \Res^{W_2}_{W_2(\tilde \lambda_0)}(g_{\lambda_0}(\hat{\eta}_0)).\\
\end{array}
$$
We then define $f_{\lambda_0}: \Irr_{2'}(W_0(\lambda_0)) \to \Irr_{2'}(W_2)$ by mapping the character $\hat{\eta}_0$ to $\widehat{f_0(\eta_0)}$ and extending this map $\Irr(W_0(\lambda_0)/W_0(\tilde{\lambda}_0))$-equivariantly. As $f_0$ is $E_0$-equivariant and $A_0$ is $E_0$-stable, so is $f_{\lambda_0}$.
	
	Finally, if $G_0$ is of type $D_4(q)$ then $W_2$ has index $3$ in $W_0$. Hence, $W_0(\lambda_0)=W_2$ or $W_0(\lambda_0)=W_0$. In the former case, we set $f_{\lambda_0}$ to be the identity map and in the latter case it is easy to explicitly construct a bijection $f_{\lambda_0}$ with the required properties.
\end{proof}

We are also interested in the action of automorphisms on local characters. To compute this action we use the following explicit parametrisation of characters.
Recall that $T:=\C_{\bG}(\bS)^{vF}=\bT^{vF}$ and $N:=\N_{\bG}(\bS)^{vF}$. 

\begin{prop}\label{parametrisation}
	Let $\Lambda$ be the extension map from \cite[Corollary 3.13]{MS} with respect to $T \lhd N$. Then the map
	$$\Pi: \mathcal{P}=\{(\lambda,\eta) \mid \lambda \in \Irr(T), \eta \in \Irr(W(\lambda)) \} \to \Irr(N), \, (\lambda, \eta) \mapsto \Ind_{N_\lambda}^{N}(\Lambda(\lambda) \eta),$$
	is surjective and satisfies
	\begin{enumerate}
		\item  $\Pi(\lambda, \eta)=\Pi(\lambda',\eta')$ if and only if there exists some $n \in N$ such that ${}^n \lambda= \lambda'$ and ${}^n \eta= \eta'$.
		\item ${}^\sigma \Pi(\lambda,\eta)=\Pi({}^\sigma \lambda, {}^\sigma \eta)$ for all $\sigma \in E$.
		\item Let $t \in \tilde T$, $\tilde \lambda \in \Irr(\tilde T \mid \lambda)$ and $\nu_t \in \Irr(N_\lambda/N_{\tilde \lambda})$ be the faithful linear character given by ${}^t \Lambda(\lambda)=\Lambda(\lambda) \nu_t$. Then we have ${}^t \Pi(\lambda,\eta)=\Pi(\lambda,\eta \nu_t)$.
	\end{enumerate}
\end{prop}

\begin{proof}
See \cite[Proposition 3.15]{MS}.
\end{proof}

Recall that $P$ is a Sylow $2$-subgroup of $N$ whose image in $N/T$ is $W_2$. We denote
$$\mathcal{P}_2=\{(\lambda,\eta) \mid \lambda \in \Irr(T_2)^{W_2}, \eta \in \Irr_{2'}(W_2)  \}.$$
As in the proof of Proposition \ref{direct product} we obtain that the map
$$\begin{array}{clcl} \Pi_{\mathrm{loc}}: & \mathcal{P}_2 &  \to &  \Irr_{2'}(P)\\
 &  (\lambda, \eta) & \mapsto &  \Res^{N_\lambda}_P(\Lambda(\lambda)) \eta,
\end{array}$$
is a bijection.

In the following we will compare the parametrisations arising in the two groups $G_0:=\bG^F$ and $G:=\bG^{vF}$.
For this, denote by $\Lambda_0$ the extension map from \cite[Corollary 3.13]{MS} with respect to $T_0 \lhd N_0$. 
To understand the action of automorphisms, recall that $D:=\iota^{-1}(E)/ \langle \hat{F} \rangle$, see the remarks after Corollary \ref{construction} and $\tilde{\bG}=\tilde{\bT}\bG$ is the regular embedding as in Section~\ref{setup}.
Thus for $ \tilde{t} \in \tilde{\bT}^F$, write $\tilde{t}=tz$ with $t \in \bT$ and $z \in \mathrm{Z}(\tilde{\bG})$. Then for $g$ from Lemma~\ref{TwistIsom} the element ${}^g \tilde{t} \in \tilde{\bT}^{vF}$ and we have a decomposition ${}^g \tilde{t}={}^g t {}^g z={}^g t z$.
%\end{rem}

\begin{prop}\label{compare extension map}
	Let $\lambda_0 \in \Irr(T_0)^{W_2}$ and set $\lambda=\alpha(\lambda_0) \in \Irr(T)^{W_2}$, for $\alpha$ from Lemma~\ref{conversion}. Suppose that $\nu_0 \in \Irr(W_0(\lambda))$ and $t_0 \in \tilde {T}_0$ satisfy ${}^{t_0} \Lambda_0(\lambda_0)=\Lambda(\lambda_0) \nu_0$. Then we have $ \Res^{N_\lambda}_P({}^{t}\Lambda(\lambda))=\Res^{N_\lambda}_P(\Lambda(\lambda)) \Res_{W_2}^{W_0(\lambda)}(\nu_0)$, where $t=\iota(t_0)$.
\end{prop}

\begin{proof}
We first recall the general construction of the extension map $\Lambda$ with respect to $T \lhd N$ (see in particular the proof of \cite[Corollary 3.13]{MS}). First one constructs an extension map $\bH \lhd \bV$. In a second step one uses this to construct an extension map $\Gamma$ with respect to $H=\bH^{vF} \lhd V=\bV^{vF}$.
%(note $V/H \cong C_W(v)$).
The extension map $\Lambda$ is then obtained by sending $\lambda \in \Irr(T)$ to the unique common extension $\Lambda(\lambda)$ in $N_\lambda=TV_\lambda$ of $\lambda$ and the restriction of $\Gamma(\Res^{T}_{H}(\lambda))$ to $V_\lambda$.

Now, let $\nu \in \Irr(W(\lambda))$ such that ${}^t \Lambda(\lambda)=\Lambda(\lambda) \nu$. For $n \in V_\lambda$ we have
$${}^t \Lambda(\lambda)(n)=\lambda([t,n]) \Lambda(\lambda)(n).$$ If $\tilde \lambda \in \Irr(\tilde{T})$ is an extension of $\lambda$ then we can write 
$\lambda([t,n])=\tilde{\lambda}(t) {}^n\tilde{\lambda}(t^{-1})$. We have ${}^n \tilde \lambda=\tilde \lambda \hat{z}$ for some linear character $\hat{z} \in \Irr(\tilde{T}/T)$. We conclude that $\nu(w)=\hat{z}(t)$, where $w$ is the image of $n$ in $W(\lambda)$. Since $\nu$ is a character of $N_\lambda/T \cong V_\lambda/H$ this uniquely determines $\nu$. Now for $w \in W_2$ the equality ${}^w \tilde \lambda= \tilde \lambda \nu$ implies by Lemma \ref{conversion2} that ${}^w \tilde{\lambda}_0=\tilde{\lambda}_0 \hat{z}$. The same reasoning as above now equally applies to the extension map $\Lambda_0$ with respect to $T_0 \lhd N_0$. Therefore, for $w \in W_2$ we find that $\nu_0(w)=\hat{z}(t_0)=\hat{z}(t)= \nu(w)$. We thus obtain
$$ \Res^{N_\lambda}_P({}^{t}\Lambda(\lambda))=\Res^{N_\lambda}_P(\Lambda(\lambda)) \Res_{W_2}^{W(\lambda)}(\nu)=\Res^{N_\lambda}_P(\Lambda(\lambda)) \Res_{W_2}^{W(\lambda)}(\nu_0),$$
which finishes the proof.
\end{proof}

We now turn to the action of automorphisms on the global characters. 
\begin{thm}\label{HC}
	Let $x\in \tilde{T}_0D$ and $\delta_{\lambda_0,x} \in \Irr(W_0({}^x \lambda_0))$ such that $\delta_{\lambda_0,x} \Lambda_0({}^x \lambda_0)={}^x \Lambda_0(\lambda_0)$. 
Then $${}^x(R_{T_0}^{G_0}(\lambda_0)_{\eta_0})=R_{T_0}^{G_0}({}^x \lambda_0)_{{}^x \eta_0 \delta_{\lambda_0,x}^{-1} }.$$
\end{thm}

\begin{proof}
Follows from the results of \cite[Theorem 5.7]{MS} as explained in the proof of \cite[Proposition 6.3]{MS}.
\end{proof} 

\begin{rem}\label{AltEquivariance}
In the following theorem, to compensate for the inversion of $\delta_{\lambda_0,x}$ occurring in Theorem~\ref{HC}, a slightly altered version of $f_{\lambda_0}$ from Lemma~\ref{equivariance} is required. 
Fix $\mathcal{T}$ a $ \Irr(W_0(\lambda_0) / W_0(\tilde \lambda_0))\rtimes E_{\lambda_0}$-transversal on $\Irr_{2'}(W_0(\lambda_0))$.
Then for $\eta_0\in\mathcal{T}$, $\sigma\in E_{\lambda_0}$ and $\mu_0\in \Irr(W_0(\lambda_0) / W_0(\tilde \lambda_0))$ define $$f_{\lambda_0}':\Irr_{2'}(W_0(\lambda_0)) \to \Irr_{2'}(W_2)$$ by setting  $f_{\lambda_0}'({}^\sigma\eta_0 \mu_0):=f_{\lambda_0}({}^\sigma\eta_0 \mu_0^{-1})$.

It follows from construction for every character $\eta_0 \in \Irr_{2'}(W_0(\lambda_0))$ and $\mu_0 \in \Irr(W_0(\lambda_0) / W_0(\tilde \lambda_0))$ then $f'_{\lambda_0}(\eta_0\mu_0)=f'_{\lambda_0}(\eta_0) \Res^{W_0(\lambda_0)}_{W_2}(\mu_0^{-1})$.
Moreover, the definition implies that $f'_{\lambda_0}$ is also an $E_{\lambda_0}$-equivariant bijection.

\end{rem}

\begin{thm}\label{AM bijection}
	Assume the setting of Section \ref{setup}.
	For $b$ and $B$ the principal $2$-block of $G_0$ respectively $\N_{G_0}(Q)$, there exists an $\N_{\tilde G_0 D}(Q)$-equivariant bijection $\kappa:\Irr_0(b) \to \Irr_0(B)$.
\end{thm}

\begin{proof}
Restriction defines an $\N_{\tilde G_0 D}(Q)$-equivariant bijection $\Irr_{0}(B) \to \Irr_{2'}(Q)$.
Additionally $\iota$ induces an equivariant bijection $\iota'$ between $\Irr_{2'}(Q)$ and $\Irr_{2'}(P)$, that is $\iota(\N_{\tilde G_0E}(Q))=\N_{\tilde GE}(P)$ and for $x\in \N_{\tilde G_0E}(Q)$, then $\iota'({}^x\chi)={}^{\iota(x)}\iota'(\chi)$.
Thus it suffices to produce an equivariant bijection $$\kappa':\Irr_0(b)\to \Irr_{2'}(P),$$ that is $\kappa'({}^x\chi)={}^{\iota(x)}\kappa'(\chi)$, for $x\in\N_{\tilde G_0 D}(Q)$ and $\chi\in \Irr_0(b)$. 
Note that as $P$ is $E$-stable, Corollary~\ref{star property} implies $\N_{\tilde GE}(P)=\C_{\tilde{G}}(\tilde{P})\tilde{P}E$ for $\tilde P=\tilde T_2 P$. 
Thus the action on $\Irr_{2'}(P)$ arises from $\tilde{T}_2E$.

By the proof of Theorem \ref{global} we have a bijection $\Pi_{\mathrm{glo}}:(\mathcal{P}_0)_2 \to \Irr_0(b)$. 
On the other hand $\Pi_{\mathrm{loc}}:\mathcal{P}_2 \to \Irr_{2'}(P)$ is a bijection.
Finally by combining Proposition \ref{conversion} and Remark~\ref{AltEquivariance} there is a bijection $(\mathcal{P}_0)_2 \to \mathcal{P}_2$ which sends a pair $(\lambda_0,\eta_0)$ to $(\alpha(\lambda_0),f'_{\lambda_0}(\eta_0))$ between parameter sets.
More explicitly, combining these yields a bijection
$$\begin{array}{clcl}
\kappa': & \Irr_0(b) & \to & \Irr_{2'}(P)\\
 & R_{T_0}^{G_0}(\lambda_0)_{\eta_0} & \mapsto & \Res_P^{N_{\alpha(\lambda_0)}}(\Lambda(\alpha(\lambda_0))) f'_{\lambda_0}(\eta_0).\\
\end{array}
$$ 

The equivariance of this bijection can be derived by combining the properties of Harish-Chandra induction established in Theorem \ref{HC}, the properties of the parametrisation from Proposition \ref{parametrisation} and Proposition \ref{compare extension map}:

Take $x\in \tilde{T}_0E$ and $\delta_{\lambda_0,x}$ such that ${}^{\tilde{t}_0}\Lambda({}^\sigma \lambda_0)= \Lambda({}^{\tilde{t}_0\sigma} \lambda_0)\delta_{\lambda_0,x}$ 
By Remark~\ref{AltEquivariance} $$
f'_{{}^x\lambda_0}({}^x\eta_0\delta_{\lambda_0,x}^{-1})=f'_{{}^x\lambda_0}({}^x\eta_0)\Res^{W_0({}^x\lambda_0)}_{W_2}(\delta_{\lambda_0,x})={}^{\iota(x)}f'_{\lambda_0}(\eta_0)\Res^{W_0({}^x\lambda_0)}_{W_2}(\delta_{\lambda_0,x}).
$$
While by Proposition \ref{parametrisation} and Proposition \ref{compare extension map} $${}^{\iota(x)}\left(\Res_P^{N_{\alpha(\lambda_0)}}(\Lambda(\alpha(\lambda_0)))\right)=\Res_P^{N_{\alpha({}^x\lambda_0)}}(\Lambda(\alpha({}^x\lambda_0)))\Res^{W_0({}^x\lambda_0)}_{W_2}(\delta_{\lambda_0,x})
$$
Thus for $x\in \tilde{T_0}E$, the equivariance follows as 
$$
\begin{array}{ccl}
\kappa'({}^xR_{T_0}^{G_0}(\lambda_0)_{\eta_0}) & = & \Res_P^{N_{\alpha({}^x\lambda_0)}}(\Lambda(\alpha({}^x\lambda_0))) f_{{}^x\lambda_0}({}^x\eta_0\delta_{\lambda_0,x}^{-1})\\
 & = & {}^{\iota(x)}\left(\Res_P^{N_{\alpha(\lambda_0)}}(\Lambda(\alpha(\lambda_0))) f_{\lambda_0}(\eta_0)\right)\\
  & = & {}^{\iota(x)}\kappa'(R_{T_0}^{G_0}(\lambda_0)_{\eta_0}).
\end{array}
$$
\end{proof}

%\begin{cor}\label{AM}
%.
%\end{cor}
%
%\begin{proof}
%	This follows from Theorem \ref{AM bijection} and Lemma \ref{normaliser}.
%\end{proof}

\section{Characters of $\tilde{G}_0$}\label{tilde G}

In order to check the inductive conditions for $G_0:=\bG^F$ we also need information on characters of $\tilde G_0=\tilde{\bG}^F$ covering characters of $2'$-degree of $G_0$.
Recall that $T:=\bT^{vF}$, $N:=\N_{\bG}(\bS)^{vF}$ and $\Lambda$ is an extension map with respect to $T\lhd N$.
Additionally $\tilde{T}:=\tilde{\bT}^{vF}$ and $\tilde{N}:=\N_{\tilde{\bG}}(\bS)^{vF}$.

\begin{prop}\label{Tilde Extension}
	There exists an $NE$-equivariant extension map $\tilde \Lambda$ with respect to $\tilde T \lhd \tilde N$ given by sending $\tilde \lambda \in \Irr(\tilde T)$ to the unique common extension of $\tilde{\lambda}$ and $\Res_{N_{\tilde \lambda}}^{N_{\lambda}}(\Lambda(\lambda))$, where $\lambda=\Res_{T}^{\tilde T}(\tilde \lambda)$.
\end{prop}

\begin{proof}
	This was shown in the proof of \cite[Proposition 3.20]{MS}.
\end{proof}

%We recall \cite[Theorem 13.9(b)]{Bonnafe2}.

%\begin{thm}\label{bonnafe}
%	Let $\mathbf{L}$ an $F$-stable Levi subgroup of $\bG$ contained in an $F$-stable parabolic subgroup of $\bG$.
%	Let $\eta \in W_G(\bL,\lambda)$ and $\tilde \eta \in W_{\tilde G}(\tilde \bL, \tilde \lambda)$. Then 
%	$$\langle R_L^G(\lambda)_\eta, \Res_{G}^{\tilde G} (R_{\tilde L}^{\tilde G}(\tilde \lambda)_{\tilde \eta}) \rangle= \langle \Res_{W_{\tilde G}(\tilde \bL, \tilde \lambda)}^{W_G(\bL,\lambda)}(\tilde \eta), \eta \rangle.$$
%\end{thm}

%Note that $W(\tilde \lambda) \lhd W(\lambda)$
%As the restriction of characters from $\tilde{G}$ to $G$ is multiplicity free, it follows that any $\eta \in \Irr(W(\lambda))$ restricts multiplicity free to $W((\tilde\lambda)$ by Theorem \ref{bonnafe}.
%see Bonnafe Corollary 13.13(a).
%We have $W_2(\tilde \lambda) \lhd W_2(\lambda)$. Since every character of $\Irr_0(W_2)$ is linear it obviously restricts irreducibly to $W_2(\lambda)$.

\begin{df}\label{cov}
We say that $(\lambda_0,\eta_0) \in (\mathcal{P}_0)_2$ (as defined in Section~\ref{GlobalChar}) is covered by the pair $(\tilde \lambda_0, \tilde \eta_0)$ if $\tilde \lambda_0 \in \Irr(\tilde T_0 \mid \lambda_0)$ and $\tilde{\eta}_0 \in \Irr(W(\tilde \lambda_0) \mid \eta_0)$. Note that $\tilde \eta_0= \Res_{W(\tilde\lambda_0)}^{W(\lambda_0)}(\eta_0)$ since $W(\lambda_0)/W(\tilde \lambda_0)$ is a $2$-group and $\eta_0$ has $2'$-degree.
\end{df}

In the proof of Theorem~\ref{AM bijection}, the set $(\mathcal{P}_0)_2$ was used to provide a bijection between the height zero characters of the principal blocks of $G_0$ and $P$ by mapping $R_{T_0}^{G_0}(\lambda_0)_{\eta_0}$ to $\Res^{N_\lambda}_P(\Lambda(\lambda)) f_{\lambda_0 }(\eta_0)$, where $f_{\lambda_0}$ is from Lemma~\ref{equivariance} and  $\lambda:=\alpha(\lambda_0)$ for $\alpha$ as defined in Section~\ref{DualityChar}. 
The notion of covering defined for $(\mathcal{P}_0)_2$ can help understand those characters which cover the height zero characters in the principal blocks of $G_0$ and $P$ under the action of $\tilde{G}_0$ respectively $\tilde{P}:=\tilde{T}_2P$.
Recall that $\tilde{\alpha}$ from Lemma~\ref{conversion2} is a bijection between $\Irr(\tilde T_0 \mid \Irr(T_0)_2^{W_2})$ and $\Irr(\tilde T \mid \Irr(T)_2^{W_2} )$.

\begin{lm}\label{cover}
	Suppose that $(\tilde \lambda_0, \tilde \eta_0)$ covers $(\lambda_0,\eta_0) \in (\mathcal{P}_0)_2$ as in Definition \ref{cov}.
	\begin{enumerate}[label=(\alph*)]
	
\item 
	Then the character $\tilde \chi:=R_{\tilde T_0}^{\tilde G_0}(\tilde \lambda_0)_{\tilde \eta_0}$ covers $\chi:=R_{T_0}^{G_0}(\lambda_0)_{\eta_0}$.
\item 
Then the character $\tilde \psi:=\Ind_{\tilde{P}_{\tilde{\lambda}}}^{\tilde P}\left(\Res^{\tilde{N}_{\tilde{\lambda}}}_{\tilde P_{\tilde{\lambda}}}(\tilde{\Lambda}(\tilde \lambda)) \Res^{W_2}_{W_2(\tilde \lambda)}(f_{\lambda_0}(\eta_0))\right)$ covers \newline $\psi:=\Res^{N_\lambda}_P(\Lambda(\lambda)) f_{\lambda_0 }(\eta_0)$, for  $\lambda:=\alpha(\lambda_0)$ and $\tilde{\lambda}:=\tilde{\alpha}(\tilde{\lambda}_0)$.
	
		\end{enumerate}
	In particular, the characters $\tilde \psi$ and $\tilde \chi$ lie above the same central character of $\Z(\tilde{G})$.
\end{lm}

\begin{proof}
Part (a) follows from \cite[Theorem 13.9(b)]{Bonnafe2}, while part (b) is a consequence of Proposition \ref{Tilde Extension} and can be obtained as in \cite[Corollary 3.21]{MS}. 
For the final statement about central characters observe that $\tilde \psi$ lies above the character $\tilde \lambda \in \Irr(\tilde T)$ and $\tilde \chi$ lies above $\tilde \lambda_0 \in \Irr(\tilde T_0)$. By the properties of the bijection in Lemma \ref{conversion2} they both lie above the same character of $\mathrm{Z}(\tilde G_0)=\mathrm{Z}(\tilde G)$.
\end{proof}

The following lemma is crucial in verifying the inductive conditions.

\begin{lm}\label{auto cover}
	Let $\tilde{\chi} \in \Irr(\tilde{G}_0)$ and $\tilde{\psi} \in \Irr(\tilde{P})$ as in Lemma \ref{cover}. Let $\sigma \in E$ and suppose that $\tilde{\chi}^\sigma= \tilde{\chi} \hat{z}$ for some $\hat{z} \in \Irr(\tilde G_0 / G_0)$. Then we have $\tilde{\psi}^\sigma= \tilde{\psi} \hat{z}$.
\end{lm}

\begin{proof}
By \cite[Corollary 6.4]{MS} there exists a character $\chi \in \Irr(G_0 \mid \tilde \chi)$ which satisfies $(\tilde{G}_0 E)_\chi=(\tilde{G}_0)_\chi E_\chi$. Therefore, we have $\chi^\sigma=\chi$ and consequently if $(\lambda_0, \eta_0)$ is the label in $(\mathcal{P}_0)_2$ of $\chi$ we have $(\lambda_0^\sigma,\eta_0^\sigma)=(\lambda_0,\eta_0)$. We have $\tilde \lambda_0^\sigma= \tilde \lambda_0 \hat{z}$ for some $\hat{z} \in \Irr(\tilde T/ T)$ and so we obtain that $W_0(\tilde \lambda_0)$ is $\sigma$-stable. Moreover, $\tilde \eta_0^\sigma=\tilde \eta_0$. We obtain $\tilde \Lambda_0(\tilde \lambda_0)^\sigma= \tilde \Lambda_0(\tilde \lambda_0) \hat{z}$, see Proposition \ref{Tilde Extension}.
Thus,
$$\tilde \chi^\sigma=R_{\tilde T_0}^{\tilde G_0}(\tilde \lambda_0^\sigma)_{\tilde \eta_0^\sigma}=
%R_{\tilde T_0}^{\tilde G}(\tilde \lambda_0^\sigma)_{\tilde \eta}
%=R_{\tilde T_0}^{\tilde G}(\tilde \lambda_0 \hat{z})_{\tilde \eta}=R_{\tilde T_0}^{\tilde G}(\tilde \lambda_0 \hat{z})_{\tilde \eta}=
\hat{z} R_{\tilde T_0}^{\tilde G_0}(\tilde \lambda_0)_{\tilde \eta_0}=\hat{z} \tilde \chi,$$
where the second to last equality is derived from \cite[Proposition 13.15]{Bonnafe2}. Moreover, $\tilde \lambda^\sigma= \tilde \lambda \hat{z}$ by Lemma \ref{conversion2}. On the other hand, $\tilde{\Lambda}(\tilde \lambda)^\sigma=\tilde{\Lambda}(\tilde \lambda) \hat{z}$ by \cite[Proposition 3.20]{MS} and so
$\tilde \psi^\sigma=\tilde\psi \hat{z},$
which finishes the proof.
\end{proof}

\section{The inductive conditions}

In this section, we show that the principal $2$-block of $G_0:=\bG^F$ for $(\bG,F)$ as in Section \ref{setup} satisfies the AM-condition. 
Recall that $Q:=\iota^{-1}(P)$ from Section~\ref{setup} is a Sylow $2$-subgroup of $G_0$.
In the following $b$ and $B$ denote the principal $2$-block of $G_0$ respectively $\N_{G_0}(Q)$.
We need the following lemma.
%The following was observed for global characters but not for "local characters".

\begin{lm}\label{all characters extend}

\,
\begin{enumerate}[label=(\alph*)]
	\item 

	Let $\chi \in \Irr_{2'}(G_0)$. Then $\chi$ extends to $G_0D_\chi$.
	\item
	Let $\chi' \in \Irr_{0}(\N_{G_0}(Q),B)$. Then $\chi'$ extends to $\N_{G_0 D}(Q)_{\chi'}$ and $\N_{\tilde{G}_0}(Q)_{\chi'}$.
\end{enumerate}
\end{lm}

\begin{proof}
The first part was proved in \cite[Proposition 8.10]{MS}. 

For the first statement of (b) we pass to the Sylow $2$-subgroup $P=\iota(Q)$ of $\bG^{vF}$. We have $\N_{G}(P)=\C_{G}(P) P$ by Theorem \ref{Malle}. In particular, any height zero character of the principal block of $\N_G(P)$ is a trivial extension of a linear character of $P$. In other words, it is enough to show that every linear character $\lambda \in \Irr(P)$ extends to a character $\hat{\lambda} \in \Irr(P E_\lambda)$ with $v\hat{F}$ in its kernel. For this choose a linear character $\nu \in \Irr(E_\lambda)$ with $\nu(\hat{F})=\lambda(v)^{-1}$ and define $\hat{\lambda}(pe):=\lambda(p) \nu(e)$ for $p \in P$ and $e \in E_\lambda$. The second part follows from Corollary \ref{star property} and Lemma \ref{cover}(b).
\end{proof}

The following lemma also helps the checking of the inductive conditions (even though we won't use it in the upcoming arguments).

\begin{lm}
Any character in $\Irr_0(b)$ or $\Irr_0(B)$ has $Z(G_0)$ in its kernel.	
\end{lm}

\begin{proof}
Let $\chi \in \Irr_{2'}(G_0)$. Since $\chi$ is of $2'$-degree and $G_0$ is quasi-simple, the character $\chi$ has $\mathrm{Z}(G_0)_{2}$ in its kernel. If $\chi$ lies moreover in the principal block, then $\chi \in \mathcal{E}(G_0,s)$ for some $2$-element $s$. Thus, $\chi$ is also trivial on $Z(G_0)_{2'}$ by \cite[Lemma 2.2]{MalleHZ}. The local height zero characters were parametrised after Proposition \ref{parametrisation}. Thus, for them the result follows from Lemma \ref{cover}. 
\end{proof}

We will use the following theorem to check the inductive condition for the blocks in question. For the language of character triples and the definition of the relation $\geq_b$ we refer the reader to \cite[Section 1.1]{Jordan2}. Moreover, for $\chi \in \Irr(H)$, an irreducible character of a finite group $H$, we denote by $\mathrm{bl}(\chi)$ the $2$-block of $H$ to which $\chi$ belongs.

\begin{thm}\label{12}
	Let $ \chi \in \Irr(G_0,b)$ and $\chi' \in \Irr(\N_{G_0}(Q),B)$ such that the following holds:
	\begin{enumerate}[label=(\roman*)]
		\item We have $(\tilde{G}_0 {D})_\chi = (\tilde{G}_0)_\chi {D}_\chi$ and 
		$\chi$ extends to $(G_0 {D})_\chi$.
		\item
		We have $( \N_{\tilde{G}_0}(Q)  \mathrm{N}_{G_0 {D}}(Q) )_{\chi'}= \N_{\tilde{G}_0}(Q)_{\chi'} \mathrm{N}_{G_0 {D}}(Q)_{\chi'}$ and $\chi'$ extends to $ \mathrm{N}_{G_0 {D}}(Q)_{\chi'}$ and $ \mathrm{N}_{\tilde{G}_0}(Q)_{\chi'}$.
		\item $(\tilde{G}_0 {D})_\chi = G_0 ( \N_{\tilde{G}_0}(Q)  \mathrm{N}_{G_0 {D}}(Q) )_{\chi'}$.
		\item There exists $\tilde{\chi} \in \Irr(\tilde{G}_0 \mid \chi)$ and $\tilde{\chi}' \in \Irr(\N_{\tilde{G}_0}(Q) \mid \chi')$ such that the following holds:
		\begin{itemize}

			\item  For all $m \in {\mathrm{N}_{G_0 {D}}(Q)}_{\chi'}$ there exists $\nu \in \Irr(\tilde{G}_0 /G_0)$ with $\tilde{\chi}^m= \nu \tilde{\chi}$ and $\tilde{\chi}'^m=\mathrm{Res}^{\tilde{G}_0}_{\N_{\tilde{G}_0}(Q)}(\nu) \tilde{\chi}'$.
			\item The characters $\tilde{\chi}$ and $\tilde{\chi}'$ cover the same underlying central character of $\mathrm{Z}(\tilde{G}_0)$.
		\end{itemize}
		\item The Clifford correspondents $\tilde{\chi}_0 \in \Irr((\tilde{G}_0)_\chi \mid \chi)$ and $\tilde{\chi}'_0 \in \Irr(\N_{\tilde{G}_0}(Q)_{\chi'} \mid \chi')$ of $\tilde{\chi}$ and $\tilde{\chi}'$ respectively satisfy $\mathrm{bl}(\tilde{\chi}_0)= \mathrm{bl}(\tilde{\chi}_0')^{(\tilde{G}_0)_\chi}$.
	\end{enumerate}
	Let $Z_0:=  \mathrm{Ker}(\chi) \cap \mathrm{Z}(G_0) $.
	Then
	$$(( \tilde{G}_0 {D})_\chi /Z_0, G_0/Z_0, \overline{\chi}) \geq_b ((\N_{\tilde{G}_0}(Q) \mathrm{N}_{G_0 {D}}(Q))_{\chi'} / Z_0,\N_{G_0}(Q)/Z_0, \overline{\chi'}),$$
	where $\overline{\chi} \in \Irr(G_0/Z_0)$ and $\overline{\chi'} \in \Irr(\N_{G_0}(Q)/Z_0)$ are the characters which inflate to $\chi$, respectively $\chi'$.
\end{thm}

\begin{proof}
	This is a consequence of \cite[Theorem 2.1]{Jordan2} and \cite[Lemma 2.2]{Jordan2}. 
\end{proof}

Note that all conditions in Theorem \ref{12} except condition (v) only depend on the character theory of $G_0$ and $\tilde{G}_0$ (together with its associated groups).

\begin{thm}\label{principal block}
	Let $(\bG,F)$ be as in Section \ref{setup}.	Then the principal $2$-block $b$ of $G_0$ satisfies the AM-condition.
\end{thm}

\begin{proof}
	We show that the bijection $\kappa: \Irr_0(b) \to \Irr_0(B)$ from Theorem \ref{AM bijection} is a strong AM-bijection in the sense of \cite[Definition 1.9]{Jordan2}. Let $\chi \in \Irr_0(b)$ and $\chi':=\kappa(\chi)$. By possibly conjugating $\chi$ by an element of $\tilde G$ we can assume by \cite[Theorem 2.11]{Jordan2} that the character $\chi$ satisfies condition (i) of Theorem \ref{12}. Using the Butterfly Theorem \cite[Theorem 1.10]{Jordan2} we see that it's enough to show that $\chi$ and $\chi'$ satisfy the remaining conditions in Theorem \ref{12}. Since $\kappa$ is equivariant we deduce that conditions (ii) and (iii) hold (the extendibility of the local character follows from Lemma \ref{all characters extend}(b)). Let $\tilde{\chi}$ and $\tilde \psi$ be the characters constructed in Lemma \ref{cover}. By Lemma \ref{auto cover} these characters satisfy condition (iv) of Theorem \ref{12}. Finally for condition (v) let $\tilde{\chi}_0 \in \Irr((\tilde{G}_0)_\chi \mid \chi)$ and $\tilde{\chi}'_0 \in \Irr(\N_{\tilde{G}_0}(Q)_{\chi'} \mid \chi')$ be the Clifford correspondents of $\tilde{\chi}$ and $\tilde{\chi}'$. 
For $\tilde{b}$ the principal $2$-block of $\tilde{G}_0$, we obtain a bijection
	$$Z(\tilde G_0)_{2'} \to \mathrm{Bl}(\tilde G_0 \mid b), \, z \mapsto \tilde{b}\otimes \hat{z},$$
between the elements of odd order in $Z(\tilde{G}_0)$ and the set of blocks of $\tilde G_0$ covering the principal block of $G$. In particular, the block of a character of $\tilde G$ covering a character in the principal block of $G_0$ is uniquely determined by its underlying character of $Z(\tilde G_0)_{2'}$. Let $\hat{z} \in \Irr(Z(\tilde G_0) \mid \tilde \chi)$. Then we deduce from this bijection that the character $\tilde \chi$ lies in the block $\tilde{b} \otimes \hat{z}$. By the Harris--Knörr correspondence we deduce that the character $\tilde \psi$ lies in the Harris--Knörr correspondent of $\tilde{b} \otimes \hat{z}$. Observe that $\mathrm{bl}(\tilde \chi_0)$ is $\tilde G_0$-stable and hence the unique block of $(\tilde G_0)_\chi$ below $\mathrm{bl}(\tilde \chi)$. Similarly, $\mathrm{bl}(\tilde \chi_0')$ is the unique block of $\mathrm{N}_{(\tilde G_0)_\chi}(Q)$ below $\mathrm{bl}(\tilde \chi')$. From this it follows that $\mathrm{bl}(\tilde{\chi}_0)= \mathrm{bl}(\tilde{\chi}_0')^{(\tilde{G}_0)_\chi}$, so condition (v) holds.
\end{proof}

 \section{The remaining finite simple groups}\label{rest}

For the remaining blocks of finite simple groups the following criterion will be helpful:

\begin{lm}\label{criterion}
Let $S$ be a finite simple non-abelian group and $\ell$ a prime. Let $b$ be an $\ell$-block of the universal covering group $\hat{S}$ of $S$ with defect group $Q$ such that $\mathrm{Out}(\hat{S})_b$ is cyclic. Assume that there exists an $\Aut(\hat{S})_b$-equivariant Alperin--McKay bijection $f:\Irr_0(\hat{S},b) \to \Irr_0(\N_{\hat{S}}(Q),B)$ preserving central characters of $\Z(\hat{S})$ and that one of the following holds:
\begin{enumerate}[label=(\roman*)]
	\item all characters of $\Irr_0(\hat{S},b)$, or $\Irr_0(\N_{\hat{S}}(Q),B)$ have $\Z(\hat{S})$ in their kernel.
	\item $\mathrm{Out}(\hat{S})_b$ is an $\ell$-group.
\end{enumerate}
Then the block $b$ satisfies the AM-condition.  
\end{lm}

\begin{proof}
We check that the conditions in \cite[Definition 4.4]{LocalRep} are satisfied. Let $X:=\hat{S}/(\mathrm{Ker(\chi)} \cap \Z(\hat{S}))$. There exists an overgroup $Y$ of $X$ such that $Y/\C_Y(X) X \cong \mathrm{Out}(X)_{\chi}$ and $Y/X$ is cyclic. Let $\chi \in \Irr_{0}(\hat{S},b)$ and $\chi':=f(\chi) \in \Irr_0(\N_{\hat{S}}(Q),B)$ considered as characters of $X$ respectively $\N_X(Q)$. Assume that we are in case (i). There exist extensions $\tilde{\chi}\in \Irr(Y \mid \chi)$ and $\tilde{\chi}'\in \Irr(\N_Y(Q) \mid f(\chi))$ such that $\bl( \tilde{\chi}')^Y=\bl(\tilde{\chi})$. As $\C_X(Y)=\Z(\hat{S})$, $\tilde{\chi}$ and $\tilde{\chi}'$ lie over the same central character of $\Z(\hat{S})$. In particular, we have
$$(Y, X, \chi) \geq_b (\N_Y(Q), \N_X(Q),f(\chi))$$
by \cite[Proposition 4.4]{LocalRep}.

In case (ii) we observe that $Y/\C_Y(X) X$ is an $\ell$-group. In particular, every block of $\C_Y(X)X$ is covered by a unique block of $Y$.
By \cite[Lemma 2.16]{LocalRep} we find $\tilde{\chi}\in \Irr(Y \mid \chi)$ and $\tilde{\chi}'\in \Irr(\N_Y(Q) \mid f(\chi))$ which lie above the same character of $\C_Y(X)$. In particular, we have $\bl(\Res^{\N_Y(Q)}_{\N_X(Q) \C_Y(X)}(\tilde{\chi}'))^{Y}=\bl(\tilde{\chi})$. By \cite[Proposition 4.4]{LocalRep} this implies that
$$(Y, X, \chi) \geq_b (\N_Y(Q), \N_X(Q),f(\chi)).$$
In both cases, the Butterfly Theorem \cite[Theorem 4.6]{LocalRep} implies that the block $b$ satisfies the AM-condition.
\end{proof}

We consider now the case excluded in Section \ref{setup}. Together with Theorem \ref{principal block} this completes the proof of Theorem \ref{principal} from the introduction.

\begin{lm}
	The principal $2$-block of $G \in \{\mathrm{Sp}_{2n}(q),{}^3 D_4(q) \}$ satisfies the AM-condition whenever $q$ is an odd power of an odd prime.
\end{lm}

\begin{proof}
	In our case $\mathrm{Out}(G)$ is cyclic and every $2'$-character lies over the trivial character of $\Z(G)$. Moreover, $\Irr_{2'}(G)=\Irr_{0}(B_0(G))$ by Remark \ref{hz characters}. Let $Q$ be a Sylow $2$-subgroup of $G$. By Lemma \ref{McKay} there exists an $\mathrm{Aut}(G)_Q$-equivariant bijection $\Irr_{2'}(G) \to \Irr_{2'}(\N_G(Q))$ preserving the underlying central characters of $\Z(G)$. In particular, the principal block of $G$ satisfies the AM-condition by Lemma \ref{criterion}.
\end{proof}

We say that a simple group $S$ is AM-good for the prime $2$ if all $2$-blocks of its universal covering group satisfy the inductive AM-condition.

\begin{lm}\label{exceptional}
Let $S$ be a simple group of Lie type defined over a field of characteristic $p \neq 2$ with exceptional Schur multiplier. Then $S$ is AM-good for the prime $2$.
\end{lm}

\begin{proof}
As argued in \cite[Proposition 14.8]{Jordan3} it suffices to consider as $S$ the simple groups ${}^2 A_3(3)$ and $B_3(3)$. Let $\hat{G}$ be the universal covering group of $S$. By \cite[Theorem 4.1]{notLie} there exists a McKay-good bijection $f:\Irr_{2'}(\hat{G}) \to \Irr_{2'}(\hat{M})$, where $\hat{M}$ is the normaliser of a Sylow $2$-subgroup of $\hat{G}$. The distribution of $2$-blocks of $\hat{G}$ is known by \cite{Breuer2}. We observe that for every character $\nu \in \Irr(\Z(\hat{G}))$ of $2'$-order there exists a unique $2$-block $b_\nu$ of $\hat{G}$ of maximal defect associated to it. Moreover, as argued in the proof of \cite[Theorem 4.1]{notLie} we have that $\mathrm{Out}(\hat{G})_\nu$ is a cyclic $2$-group for every $1 \neq \nu$ of $2'$-order. The principal block satisfies the AM-condition by Theorem \ref{principal block}. As a McKay-good bijection preserves central characters we see that $f$ preserves the block decomposition. We deduce that $b_\nu$, $\nu \neq 1$, also satisfies the AM-condition by Lemma \ref{criterion}. In particular, by \cite{Breuer2} the group ${}^2 A_3(3)$ satisfies the AM-condition, as all blocks with non-maximal defect are of central defect. We are left to consider the three blocks $b$ of the universal covering  group $\hat{G}$ of $B_3(3)$ of defect $4$. Let $b$ be one of these blocks. We can use a proof similar to \cite[Proposition 14.6]{Jordan3}. An inspection of \cite{Breuer2} shows that $|\Irr_0(b)|=4$. Moreover, these characters have $\Z(\hat{G})_2$ in their kernel. By \cite[Theorem 4.1]{defect4} we deduce that the Brauer correspondent $B$ of $b$ has also exactly four height zero characters which all have $\Z(\hat{G})_2$ in their kernel. Let $\bar{b}$ and $\bar{B}$ be the images of the blocks $b$ and $B$ in the maximal 3-cover $\hat{G}'$ of $S$. As $\bar{b}$ has defect $2^3$ and precisely $5$ ordinary characters (see \cite{Breuer2}) its defect group is isomorphic to the dihedral group $\mathrm{Dih}_8$, see \cite[Theorem 8.1]{Sambale}. Using \cite[Proposition 14.4, Proposition 14.5]{Jordan3} we deduce that there exists an $\mathrm{Aut}(\hat{G}')$-equivariant bijection $\Irr(\bar{b}) \to  \Irr(\bar{B})$. Thus, $b$ satisfies the AM-condition by Lemma \ref{criterion}.
\end{proof}

\begin{lm}\label{defining}
Let $S$ be a simple group of Lie type defined over a field of characteristic $p$. Then $S$ is AM-good for the prime $p$.
\end{lm}

\begin{proof}
Let $\hat{G}$ be the universal covering group of $S$ and $G$ the $p'$-cover of $S$. By \cite[Theorem 9.10]{NavarroBook} there exists a bijection between the set of $p$-blocks of $\hat{G}$ and the set of $p$-blocks of $G$. With this observation the statement follows as in the proof of \cite[Theorem 8.4]{IAM}.
\end{proof}

\section{Consequences}

In this section we derive some consequences of Theorem \ref{principal block}. We keep the notation and setup of Section \ref{setup} but we make no restriction on the type of $G$. 

\begin{cor}\label{classical}
Assume that the root system of $\bG$ is of classical type. Then every $2$-block of $G$ satisfies the AM-condition.
\end{cor}

\begin{proof}
	By Lemma \ref{exceptional} we can assume that $S:=G/\Z(G)$ has non-exceptional Schur multiplier.
	Observe that every subgraph of a Dynkin diagram of classical type is again of classical type.
	According to \cite[Theorem 3.12]{Jordan2} it suffices to prove that all strictly quasi-isolated $2$-blocks $b$ of $G$ are AM-good. Suppose first that $\bG$ is not of type $A$. Using the classification of quasi-isolated elements in \cite{Bonnafe} together with \cite[Theorem 21.14]{MarcBook} we deduce that $b$ is the principal block of $G$. The claim follows therefore from Theorem \ref{principal block}. Suppose therefore now that $\bG$ is of type $A$. Using the proof of the reduction theorem in \cite[Theorem 13.4]{Jordan3} together with the results of \cite[Section 3.3]{Jordan2} we see that it is again sufficient to prove the claim whenever $b$ is the principal block of $G$. This again follows from Theorem \ref{principal block}.
\end{proof}

\begin{cor}\label{exceptional Lie}
Suppose that the root system of $\bG$ is of exceptional type and let $b$ be a quasi-isolated $2$-block of $G$ of maximal defect. Then $b$ is the principal block of $G$.
\end{cor}

\begin{proof}
	Suppose first that $b$ is a unipotent block of $G$. Then the claim of the corollary follows from the description of defect groups given in \cite{Enguehard}.
%	Exceptions are the blocks in $E_7$ (general case of the one considered in \cite{typeA}) and blocks in type $E_8$. Both of them have non-maximal defect 
Suppose now that $G$ is not of type $E_6$. Any block of maximal defect contains a character of $2'$-degree. According to Remark \ref{hz characters} such characters lie in a unipotent block.
%Hence, except for type $E_6$ we only need to consider unipotent blocks. Alternatively, the result can be derived from \cite[Theorem 6]{2Sylow}.
	
Finally for $G=E_6( \pm q)$ the non-unipotent quasi-isolated $2$-blocks are given in \cite[Table 3]{KessarMalle}. The order of the defect group is bounded by $|C_{G^\ast}(s)|_2$, see \cite[Lemma 2.6(a)]{KessarMalle}, where $1 \neq s \in G^\ast$ is the semisimple quasi-isolated element of $2'$-order associated to the block $b$. Going through the list given in \cite{KessarMalle} one checks that $|C_{G^\ast}(s)|_2$ is always smaller than $|G|_2$.
%For this we determine the $2$-part of the groups under consideration :
%\begin{itemize}
%	\item $|A_2(q)|_2=\Phi_1^2 \Phi_2$ 
%	\item $|{}^2 A_2(q)|_2=\Phi_1 \Phi_2^2$
%	\item $|D_4(q)|_2=\Phi_1^4 \Phi_2^4 \Phi_4^2$
%	\item $|{}^2 D_4(q)|=|\Phi_1^3 \Phi_2^3 \Phi_4 \Phi_8$
%	\item $|{}^3 D_4(q)|_2=\Phi_1^2 \Phi_2^2$
%	\item $|E_6(q)|_2=\Phi_1^6 \Phi_2^4 \Phi_4^2 \Phi_8$
%\end{itemize}
%Judging from these numbers one sees that none of the centralizers has the needed $2$-part. The closest is the centralizer of order $\Phi_1^2 D_4(q)$ whose $2$-part is have of the $2$-part of $E_6(q)$ by these considerations.
%The twisted group case can then be obtained by Ennola duality as explained in \cite{KessarMalle} after the proof of Proposition 4.3.
%only powers of 2 are relevant
% the following observations help:
%	$\Phi_1(q)$ has the same $2$-part as $\Phi_1(q^i)$ for odd $i$. Ennola duality exchanges the order of twisted type to not twisted type. If $e$ order of $q$ mod $4$ and $d \in \{1,2\}$ but $d \neq e$ then $2 \mid \mid \Phi_d(q)$.
%	$\Phi_1,\Phi_2,\Phi_4,\Phi_8$ are the only cyclotomic polynomials that occur with nontrivial $2$-part.
\end{proof}

We can now complete the proof of Theorem \ref{main} from the introduction. 

\begin{thm}
	The Alperin--McKay conjecture holds for $2$-blocks of maximal defect.
\end{thm}

\begin{proof}
	By \cite[Proposition 2.5]{CS14} it suffices to establish that every block $b$ of maximal defect of the universal central extension of a finite simple non-abelian group $S$ satisfies the AM-condition. As explained in the proof of \cite[Proposition 14.8]{Jordan3} alternating groups, Suzuki and Ree groups and sporadic groups are AM-good. By Lemma \ref{exceptional} and Lemma \ref{defining} we can therefore assume that $S=G/Z(G)$, such that $G$ is a group of Lie type defined over a field of odd characteristic and $G$ is the universal covering group of $S$. By Corollary \ref{classical} we can assume that $G$ is an exceptional group of Lie type. By the main result of \cite{Jordan2} we can assume that $b$ is a quasi-isolated block. In this case the result follows from Corollary \ref{exceptional Lie} and Theorem \ref{principal block}.
\end{proof}

\bibliographystyle{plain}

\end{document}